\documentclass[a4paper,10pt]{amsart}

\usepackage[utf8]{inputenc}
\usepackage{amsmath, amsthm, amssymb, comment}
\usepackage[dvips]{graphicx}
\usepackage{overpic}
\usepackage{enumitem}
\usepackage[colorlinks=true]{hyperref}
\usepackage{pinlabel}

\usepackage[margin=3cm]{geometry}

\usepackage{color}
\newtheorem{theorem}{Theorem}[section]
\newtheorem{corollary}[theorem]{Corollary}
\newtheorem{proposition}[theorem]{Proposition}
\newtheorem{definition}[theorem]{Definition}
\newtheorem{lemma}[theorem]{Lemma}
\newtheorem{claim}[theorem]{Claim}
\newtheorem{remark}[theorem]{Remark}

\newtheorem*{theorem*}{Theorem}
\newtheorem*{proposition*}{Proposition}
\newtheorem*{definition*}{Definition}
\newtheorem*{lemma*}{Lemma}
\newtheorem*{claim*}{Claim}
\newtheorem*{corollary*}{Corollary}
\newtheorem*{convention*}{Convention}
\newtheorem{observation}[theorem]{Observation}

\theoremstyle{definition}

\newtheorem{example}[theorem]{Example}

\newtheorem{construction}[theorem]{Construction}

\theoremstyle{remark}
\newtheorem{rem}[theorem]{Remark}
\newtheorem*{rem*}{Remark}

\newtheorem{case}{Case}

\newcommand{\wt}[1]{\widetilde{#1}}

\newcommand\bR{\mathbb R}

\newcommand\bZ{\mathbb Z}

\newcommand{\R}{\mathbb R}

\newcommand\orb{ \mathcal O }

\newcommand{\cC}{\mathcal{C}}

\newcommand{\cF}{\mathcal{F}}
\newcommand{\cL}{\mathcal{L}}

\newcommand{\cP}{\mathcal{P}}

\newcommand{\cT}{\mathcal{T}}

\newcommand{\wP}{\widetilde{P}}

\newcommand{\fix}{\cP}

\newcounter{notes}%[page]   %Le 2eme argument fait reinitialiser les numeros de notes a chaque page
\title{Anosov flows with the same periodic orbits}

\author[Thomas Barthelm\'e]{Thomas Barthelm\'e}
\address{Queen's University, Kingston, Ontario}
\email{thomas.barthelme@queensu.ca}
\urladdr{sites.google.com/site/thomasbarthelme}

\author[Sergio Fenley]{Sergio Fenley}
 \address{Florida State University, Tallahassee, FL}
 \email{fenley@math.fsu.edu}
\author[Kathryn Mann]{Kathryn Mann}
 \address{Cornell University, Ithaca, NY}
 \email{k.mann@cornell.edu}
\urladdr{https://e.math.cornell.edu/people/mann}

\begin{document}

\begin{abstract}
In \cite{BFM}, it was proved that transitive pseudo-Anosov flows on any closed 3-manifold are determined up to orbit equivalence by the set of free homotopy classes represented by periodic orbits, provided their orbit space does not contain a feature
called a ``tree of scalloped regions."  In this article we describe what happens in these exceptional cases: we show what topological features in the manifold correspond to trees of scalloped regions, completely classify the flows which do have the same free homotopy data, and construct explicit examples of flows with the same free homotopy data that are not orbit equivalent.  
\end{abstract}

\maketitle

\section{Introduction}
This article concerns the classification of Anosov flows on 3-manifolds, a problem which dates back to Smale \cite{Smale}.  Since flows may always be reparameterized, the relevant notion of classification is up to {\em orbit} or {\em isotopy equivalence}.  Flows $\phi_1$ and $\phi_2$ on a manifold $M$ are orbit equivalent if there exists a homeomorphism of $M$ sending orbits of one to orbits of the other, and isotopically equivalent if the homeomorphism can be taken to be isotopic to the identity.\footnote{Here as in \cite{BFM}, we do not ask for the orbit equivalence to preserve the direction of orbits.}

The classification problem is particularly interesting on 3-manifolds, where we have both large families of examples and a rich structure theory.  
In \cite{BFM} it was shown that transitive (pseudo)-Anosov flows are ``essentially" classified by the set of free homotopy classes of their periodic orbits.   More specifically, \cite[Theorem 1.3]{BFM} gives a complete invariant of such flows up to orbit equivalence, which in many cases reduces to simply knowing this free homotopy data. 
In the present work, we make the notion of ``essentially classified by" precise, describing exactly which Anosov flows on which manifolds have the same free homotopy classes of periodic orbits, and which topological feature in $M$ is responsible for the additional invariant beyond periodic orbits needed in \cite{BFM}, thus completing the classification program started there.  

For an Anosov flow $\phi$ on a closed three-manifold $M$, let $\fix(\phi)$ denote the set of unoriented free homotopy classes of periodic orbits, that is 
\[ \fix(\phi) = \{ [\gamma] : \gamma \text{ or } \gamma^{-1} \text{ is freely homotopic to a periodic orbit} \} \]
where $[\gamma]$ denotes a conjugacy class in $\pi_1(M)$.   
We show the following. 
\begin{theorem}[Classifying Anosov flows by periodic orbits]
\label{thm_classification}
Let $\phi_1$, $\phi_2$ be Anosov flows on a closed manifold $M$, at least one of which is transitive. If $\fix(\phi_1)=\fix(\phi_2)$, then both flows are transitive, and there exists a finite (possibly empty) collection of Seifert pieces 
of the JSJ decomposition of $M$ which are {\em scalloped} and {\em periodic} for both $\phi_1$ and $\phi_2$ and such that (up to isotopy equivalence) $\phi_2$ is obtained from $\phi_1$ by a {\em periodic Seifert flip} on each piece.   

Conversely, any Anosov flow $\psi$ obtained from an Anosov flow $\phi$ by applying periodic Seifert flips to scalloped periodic Seifert pieces satisfies $\fix(\phi) = \fix(\psi)$.  
\end{theorem}

A {\em scalloped periodic Seifert piece} for a flow $\phi$ on $M$ is a piece $P$ of the JSJ decomposition of $M$ such that the free homotopy classes of a regular fiber is represented by a periodic orbit of $\phi$ {\em and} for each boundary surface of $P$ there is an independent element (not a power of the fiber) of its fundamental group represented by a periodic orbit of $\phi$.   Very roughly, a {\em periodic Seifert flip} is a cut-and-paste construction that reverses the direction of the flow along the periodic orbits representing the fiber in a periodic Seifert piece.  We describe this construction in detail, show that in the transitive case it gives a unique flow up to isotopy equivalence (Theorem \ref{thm:flip_unique}), and show that the construction may be applied to any periodic Seifert piece of an Anosov flow (Theorem \ref{thm:existence}).  The precise definition of periodic Seifert flip uses the framework of {\em spines} from \cite{BarbFen_pA_toroidal}, we defer this to Section \ref{sec:spine} -- see Definition \ref{def:flip}.  A consequence of our work (together with \cite{BFM}) is that applying a flip to the same piece twice results in a flow isotopically equivalent to the original, hence the name ``flip".  

\begin{rem}
Though stated in terms of isotopy equivalence, Theorem \ref{thm_classification} has a counterpart classification of flows up to orbit equivalence.  
In general, since 3-manifolds admitting Anosov flows are $K(\pi, 1)$ spaces, any automorphism of $\pi_1(M)$ is realized by a diffeomorphism.  Thus, one has  
$\fix(\phi_1)=\theta\fix(\phi_2)$ for some map $\theta$ on free homotopy classes induced by an automorphism of $\pi_1(M)$ if and only if there exists a diffeomorphism $\Theta$ 
of $M$ such that $\fix(\phi_1)=\fix(\Theta\circ\phi_2\circ \Theta^{-1})$; and $\Theta$ can be taken to induce $\theta$.  
Thus, by composing with automorphisms of $\pi_1(M)$, Theorem \ref{thm_classification} classifies Anosov flows up to orbit equivalence. 
\end{rem}

The terminology ``scalloped" comes from the feature of the orbit space of a flow called a {\em tree of scalloped regions}, which is the feature responsible for the additional invariants beyond $\fix(\phi)$ given in \cite{BFM}.  (See Section \ref{sec:orbit_space} for precise definitions.)  A second goal of this work is to show that these regions in the orbit space correspond to scalloped periodic Seifert pieces of $M$: 

\begin{theorem} \label{thm_scalloped_tree_iff_scalloped_piece}
The orbit space of an Anosov flow 
$\phi$ contains a tree of scalloped regions if and only if some Seifert piece $P$ of the JSJ decomposition of $M$ has both its fiber, and an independent element of the fundamental group of each boundary component represented by periodic orbits of $\phi$.   
Moreover, each such piece $P$ corresponds to a unique (up to the action of $\pi_1(M)$) tree of scalloped regions.
\end{theorem} 

Finally, we also prove that $\fix(\phi)$ alone is not a complete invariant.   Though not explicitly stated there, from \cite{BFM} one can obtain an upper bound (depending on the topology of $M$) on the number of Anosov flows on $M$ with the same set of free homotopy classes of periodic orbits.  Here, we prove this bound is sharp.  
\begin{corollary}\label{cor_sharpupperbound}
Given a manifold $M$ with $k$ distinct Seifert pieces, there are at most $2^k$ non-orbit equivalent transitive Anosov flows with a given set of free homotopy classes represented by periodic orbits.

Conversely, for every $k \in \mathbb{N}$ there are examples of manifolds $M$ with exactly $k$ distinct Seifert pieces which support $2^k$ non-orbit equivalent transitive Anosov flows with identical sets of free homotopy classes of periodic orbits.
\end{corollary}

One nuance in this result is that, while performing a Seifert flip always results in a flow which is not isotopy equivalent to the original flow (see Proposition \ref{prop:not_isotopy_eq}), showing that the resulting flow is not {\em orbit} equivalent to the original is not straightforward.  In fact, there exist examples where, due to some hidden symmetry, applying a periodic Seifert flip to a flow $\phi$ results in a flow which {\em is} orbit equivalent to $\phi$.  See Example \ref{ex:seif_flip_flow}.

\subsection{Pseudo-Anosov and non-transitive flows}\label{sec_pseudo_Anosov_case}
We have chosen to limit the scope of much of this article to Anosov flows, instead of pseudo-Anosov flows as was treated in \cite{BFM}.  We do this for only one reason: in order to build a flip on a scalloped periodic piece of an Anosov flow, we use the gluing theorem of \cite{BBY}. While their theorem is likely to generalize to the pseudo-Anosov setting, such a generalization is not immediate.  As it is beyond the scope of this article to provide such a generalization, we instead restrict ourselves to Anosov flows.  Once a generalization of \cite{BBY} for pseudo-Anosov flows is known, all the results of the present article will automatically extend to that setting.

Nevertheless, several of the proofs we do give apply already to pseudo-Anosov flows.  
For the convenience of the reader, we state here the most general versions of the results already obtained.  

The definitions of scalloped periodic Seifert pieces and flips make sense for pseudo-Anosov flows, and we use this level of generality throughout 
Section \ref{sec:orbit_space}.  
Theorem \ref{thm_scalloped_tree_iff_scalloped_piece} holds for pseudo-Anosov flows (transitive or not); this is the statement of Theorem \ref{thm:tree_iff_periodic} proved in Section \ref{sec_scalloped_periodic}.  

The proof that two periodic flips of each other have the same free homotopy data holds for pseudo-Anosov flows with no changes (see Proposition \ref{prop_same_spectra}), and that proof does not require the flows to be transitive either. As explained above, it is for the proof of the existence of the flip that we use \cite{BBY}. However, for \emph{totally periodic pseudo-Anosov flows}, which are pseudo-Anosov flows on graph-manifolds such that each of the Seifert pieces has its fiber direction represented by a periodic orbit, the work of Barbot and Fenley \cite{BarbFen_pA_toroidal,BF_totally_per} allows one to construct a flip explicitly on any piece without using \cite{BBY}. In particular, in this setting we do not need the transitivity assumption either -- we use transitivity in the proof of Proposition \ref{prop_existence_of_flip} in order to be able to easily apply the statement of \cite[Theorem 1.5]{BBY}, but see Remark \ref{rem:non_transitive_BBY}.  

Thus, we have the following existence result
\begin{theorem} If $\phi$ is either an Anosov flow or a totally periodic pseudo-Anosov flow (not necessarily transitive) with a scalloped periodic Seifert piece $P$, then there exists an Anosov, or pseudo-Anosov, flow $\psi$ obtained from $\phi$ by a periodic Seifert flip on $P$.
\end{theorem} 
We also have the following generalization of (each direction of) Theorem \ref{thm_classification}. 

\begin{theorem}
Let $\phi$ and $\psi$ be two (not necessarily transitive) pseudo-Anosov flows that are periodic flips of each other.  Then $\fix(\phi)=\fix(\psi)$.  
\end{theorem}

\begin{theorem}
Let $\phi_1$, $\phi_2$ be either Anosov flows on $M$, or totally periodic pseudo-Anosov flows, at least one of which is transitive. If $\fix(\phi_1)=\fix(\phi_2)$ then $\phi_1$ and $\phi_2$ are isotopically equivalent or obtained from one another by periodic Seifert flips as in Theorem \ref{thm_classification}. 
\end{theorem}
In fact, a careful reading of \cite{BFM} shows that in the above result, it is enough to ask for one of the flows to be Anosov or one of the flows to be totally periodic, because equality of the free homotopy data will force both flows to be Anosov or both flows to be totally periodic. 
More precisely, \cite[Section 5]{BFM} uses only the free homotopy data of pseudo-Anosov flows to construct a natural map between the associated orbit spaces, whenever this free homotopy data agrees.  Though this map may have discontinuities, it preserves adjacency of $g$-invariant lozenges for any $g \in \pi_1(M)$, which is enough to deduce whether there are singular points or a tree of scalloped regions in the orbit space.

\subsection*{Outline} 
Section~\ref{sec:orbit_space} discusses background on the orbit space of a pseudo-Anosov flow and how features of the orbit space correspond to topological features in the supporting 3-manifold.  This includes the new result that embedded scalloped surfaces are always isotopic to cutting surfaces of the JSJ decomposition.   We pay special attention to manifolds which are not necessarily orientable, filling in gaps in the literature in this case. 
The section concludes with the proof of Theorem \ref{thm_scalloped_tree_iff_scalloped_piece}.   

Section~\ref{sec:spine} extends the notion of {\em spine} to periodic pieces of pseudo-Anosov flows in non-orientable manifolds, and using this gives the precise definition of periodic Seifert flip and statements of existence and uniqueness.   Sections~\ref{sec_model_flow} to \ref{sec:per_orbits} are devoted to the constructions of flips and the proof of Theorem \ref{thm_classification}.  

Section~\ref{sec_model_flow} constructs pairs of (partially defined) ``flipped flows" on elementary models representing neighborhoods of unions of weakly embedded Birkhoff annuli, and prove that every neighborhood of a spine of an Anosov flow in a periodic piece is represented by such a model.  
In Section~\ref{sec:build_flips} we show how to glue these models along boundary to obtain Anosov flows, and in Section \ref{sec:per_orbits} we use the results of \cite{BFM} to show uniqueness (up to isotopy) of flips and finish the proof of Theorem \ref{thm_classification}.  
Finally, in Section \ref{sec:equiv_flips} we discuss the question of isotopy and orbit equivalence of flips and give examples.  

\subsection*{Acknowledgements} 
TB was partially supported by the NSERC (Funding reference number RGPIN-2017-04592). SF was partially supported by NSF DMS-2054909. KM was partially supported by NSF CAREER grant DMS-1933598 and a Sloan fellowship.

%----------%----------%----------%----------%----------
\section{Topological features of $M$ and their counterparts in the orbit space} \label{sec:orbit_space}

Throughout this section we assume $M$ is a compact 3-manifold equipped with a pseudo-Anosov flow $\phi$.  We do not recall the definitions of Anosov or pseudo-Anosov flows here, one can refer e.g., to \cite{FH_book} or \cite{Calegari_book} for a general introduction.  In particular, $M$ is irreducible.    

\subsection{Structures in the orbit space.}
The {\em orbit space} of a flow $\phi$ on $M$ is obtained by lifting $\phi$ to a flow $\wt{\phi}$ on the universal cover $\wt M$, and then passing to the quotient $\orb_\phi : = \wt{M} / \sim$ where each orbit of $\wt{\phi}$ is collapsed to a point. 
It is a fundamental theorem that, for pseudo-Anosov flows on $3$-manifolds, $\orb_\phi$ is always homeomorphic to $\bR^2$ (\cite{Fen_Anosov_flow_3_manifolds,FenMosher}). The lifts of the stable and unstable foliations of $\phi$ project to one-dimensional, transverse foliations $\cF^{s}$ and $\cF^u$ of $\orb_\phi$, with isolated prong singularities if the flow is pseudo-Anosov rather than genuinely Anosov.  

The action of $\pi_1(M)$ on $\wt{M}$ descends to an action on $\orb_\phi$ by homeomorphisms preserving $\cF^{s}$ and $\cF^u$.   Under this action, fixed points in $\orb_\phi$ of a nontrivial element $g \in \pi_1(M)$ are exactly the projections to $\orb_\phi$ of the lifts of periodic orbits in the same free homotopy class as $g$ or $g^{-1}$.     

Much more can be read off the action of $\pi_1(M)$ on $\orb_\phi$.  In fact, by \cite[Theorem B]{Bar_caracterisation} (stated for Anosov flows, but the proof applies also to pseudo-Anosov flows), the action of $\pi_1(M)$ on $\orb_\phi$ completely determines the isotopy equivalence class of the flow.  Thus, a priori, everything about a flow can be recovered from the structure of the orbit space and action of $\pi_1(M)$, although in practice this is quite difficult to do.  Nevertheless, the orbit space has become an essential tool and we begin by quickly recalling some essential structural features and their relationship with topological features of $M$ from previous work of Barbot and Fenley.  Further background and some illustrations can be found in \cite{BFM} and references therein, or the general surveys \cite{Bar_HDR,B_mini_course}. 
		
 \begin{definition}\label{def_perfect_fit}
 Two leaves (or half leaves) $l^s$ in $\cF^s$ and $l^u$ in $\cF^u$ are said to make a \emph{perfect fit} if they have empty intersection, but  ``just miss'' each other. That is, there is an arc $\tau^s$ starting at a point of $l^s$ and transverse to $\cF^s$ and an arc $\tau^u$ starting at $l^u$ transverse to $\cF^u$ such that
\begin{enumerate}[label=(\roman*)]                                            
\item every leaf $k^s$ of $\cF^s$ that intersects the interior of ${\tau}^s$ intersects $l^u$, and
\item every leaf $k^u$ of $\cF^u$ that intersects the interior of ${\tau}^u$ intersects $l^s$.
 \end{enumerate}
 \end{definition}
 
 \begin{definition}
A \emph{lozenge $L$ with corners $x$ and $y$} is the open subset of $\orb_\phi$ ``bounded" by two pairs of half-leaves starting at $x$ and $y$ making perfect fits.  Precisely, let $x,y$ be two orbits in $\orb_\phi$ and suppose $r_x^{s,u}$ and $r_y^{s,u}$ are half leaves, where the subscript denotes the starting point and the superscript their respective foliation, such that $r_x^{s}$ and $r_y^{u}$ make a perfect fit, as does  $r_x^{u}$ and $r_y^{s}$. Then 
\begin{equation*}
 L := \lbrace p \in \orb_\phi \mid \cF^u(p) \cap r_x^s \neq \emptyset \text{ and } \cF^s(p) \cap r_x^u \neq \emptyset \rbrace
\end{equation*}
is called a {\em lozenge},   
the half-leaves $r_x^{s,u}$ and $r_y^{s,u}$ are called its \emph{sides}, and $x$ and $y$ the {\em corners}. 
A \emph{closed lozenge} is the union of a lozenge with its sides and corners. 
\end{definition}

\begin{definition}
A \emph{chain of lozenges} is a union of closed lozenges that satisfies the following connectedness property: for any two lozenges $L,L'$ in the chain, there exist lozenges $L_0, \dots,L_n$ in the chain such that $L=L_0$, $L'=L_n$, and, for all $i$, $L_i$ and $L_{i+1}$ share a corner (and may or may not share a side).
We say a chain of lozenge is \emph{maximal} if it is not contained in a strictly larger chain.
\end{definition}
We often pass back and forth between $\wt{M}$ and $\orb_\phi$, using ``lozenge" also to mean the union of orbits in $\wt{M}$ that project to a lozenge in $\orb_\phi$.  

Lozenges and chains of lozenges play a fundamental role because of the following  results of Fenley. 

\begin{proposition}[Fixed points are in chains of lozenges, Theorem 3.3 of \cite{Fenley_QGAF}]  \label{prop:fixed_chain}
 If $x \neq y\in \orb_\phi$ are both fixed by some nontrivial element $g\in \pi_1(M)$, then there exists a chain of lozenges fixed by $g$ and containing both $x$ and $y$ as corners.
 \end{proposition}
 Translating this picture to $\wt{M}$, this says that lifts of freely homotopic periodic orbits in $M$ are connected by a chain of lozenges. 
 
 \begin{rem}\label{rem_maximal_chain}
 Notice that one consequence of this result is that if a non-trivial element $g\in \pi_1(M)$ fixes at least two points in $\orb_\phi$, then it fixes a unique maximal chain of lozenges.
 \end{rem}
 
For the next statement, recall that leaves of a foliation are {\em non-separated} if they represent non-separated points in the leaf space of the foliation.  
 
\begin{proposition}[Nonseparated leaves are in lines of lozenges, Theorem D of \cite{Fenley_structure_branching}] \label{prop:nonsep_implies_fixed}
 If two leaves $l_1, l_2$ in $\cF^{s}$ or  $\cF^{u}$ are non-separated then there exists a nontrivial element $g\in \pi_1(M)$ fixing both. 
\end{proposition}
This is proved for Anosov flows in \cite{Fenley_structure_branching}, but the proof carries through directly in the more general pseudo-Anosov case.  

We call a chain of lozenges $\cC$ a \emph{line of lozenges} if there exists a single leaf of $\cF^s$ or $\cF^u$ that intersects every lozenge in $\cC$. Equivalently, a chain of lozenges is a line if any two adjacent lozenges share a stable side, or any two adjacent lozenge share an unstable side.
Combining Propositions \ref{prop:fixed_chain} and \ref{prop:nonsep_implies_fixed} with the fact that any $g$-invariant leaf contains a fixed point for $g$, one sees easily that any pair of nonseparated leaves always lies in a line of lozenges fixed by some nontrivial element $g$.  

In \cite[Theorem 5.2]{Fenley_structure_branching}, Fenley shows that any infinite line of lozenges is in fact contained in a bi-infinite line that has a special structure called a {\em scalloped region} described below (again, the proof applies generally to pseudo-Anosov flows though the theorem is stated for Anosov, see also \cite[Lemma 2.32]{BFM}).   Though technical, the definition is simply describing the picture shown in Figure \ref{fig_scalloped_region} of a trivially foliated open region of $\orb_\phi$ that can be expressed as a bi-infinite line of lozenges in two different ways.  

 \begin{definition}[Scalloped region]\label{def_scalloped_region}
  A \emph{scalloped region} is an open, unbounded set $U \subset \orb_\phi$ with the following properties:
  \begin{enumerate}[label=(\roman*)]
   \item The boundary $\partial U$ consists of the union of two families of stable leaves $l_k^{1,s}, l_k^{2,s}$ and two families of unstable leaves $l_k^{1,u},l_k^{2,u}$, indexed by $k\in \bZ$.
   \item The leaves of each family $l_k^{i,\sigma}$, $k \in \bZ$, $\sigma =s,u$ are pairwise non-separated.
   \item The boundary leaves are ordered so that there exists a (unique) leaf $f_k^{1,s}$ that makes a perfect fit with $l_k^{1,u}$ and $l_{k+1}^{1,u}$. Moreover, the sequence of leaves $f_k^{1,s}$ accumulates on the leaves $\cup_{i\in \bZ} l_i^{1,s}$ as $k \to \infty$, and  on $\cup_{i\in \bZ} l_i^{2,s}$ as $k \to -\infty$.  The analogous statement holds for leaves making perfect fits with the other families $l_k^{i,\sigma}$.
   \item The bifoliation is trivial inside $U$, i.e., for all $x \neq y\in U$, $\cF^s(x)\cap \cF^u(y) \neq \emptyset $ and  $\cF^s(y)\cap \cF^u(x) \neq \emptyset $ and $U$ contains no singular points.
  \end{enumerate}
 \end{definition}
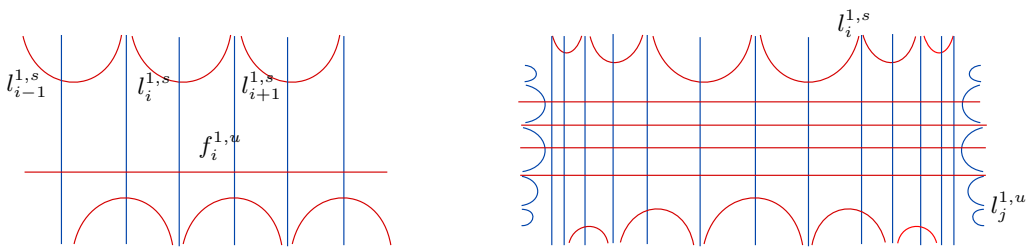
\begin{figure}[h]
   \labellist 
  \small\hair 2pt
     \pinlabel $l_{i-1}^{1,s}$ at 8 85 
     \pinlabel $l_{i}^{1,s}$ at 70 84 
   \pinlabel $l_{i+1}^{1,s}$ at 122 84
    \pinlabel $f_i^{1,u}$ at 102 55
  \pinlabel $l^{1,u}_j$ at 485 25  
  \pinlabel $l_{i}^{1,s}$ at 410 115 
 \endlabellist
     \centerline{ \mbox{
\includegraphics[width=13cm]{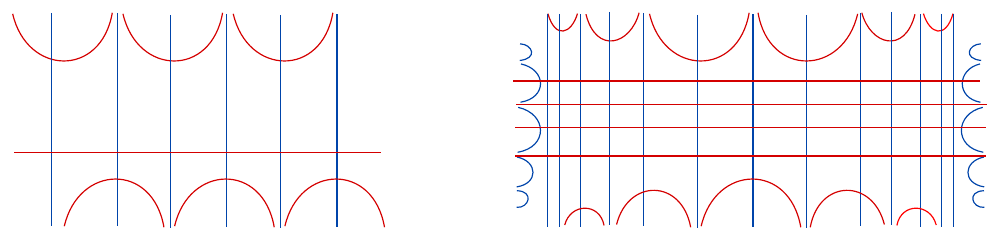} }}
\caption{A line of adjacent lozenges and a scalloped region.  The leaves $f_i^{j,u}$ are vertical (blue), $f_i^{j,s}$ are horizontal (red).}
\label{fig_scalloped_region}
\end{figure}

A scalloped region is associated with exactly two distinct infinite lines
of lozenges: The lozenges in one of them all intersect a common stable leaf,
and in the other chain of lozenges they intersect a common
unstable leaf.
The scalloped region is the union of the lozenges in one of 
these chains plus the half leaves which are contained in the
closure of ``consecutive" lozenges in such a chain.
We say that $L$ is a lozenge in a scalloped region if it is
a lozenge of one of the two such chains of lozenges.
Equivalently $L$ is a lozenge in a scalloped region $U$
if $L$ is contained in $U$.

We use the following elementary lemma on the structure of scalloped regions. 

\begin{lemma} \label{lem:intersection_scalloped}
Let $U$, $U'$ be scalloped regions with $U \cap U' \neq \emptyset$.  If $U \neq U'$, then their intersection is equal to $L \cap L'$ for some well defined,
unique lozenges $L$ of $U$ and $L'$ of $U'$.  
As a special case, if $U' = \alpha(U)\neq U$ for some $\alpha \in G$, then $L \neq L'$ and $\alpha$ has a unique fixed point in $P$, which lies in $L \cap L'$. 
\end{lemma} 

\begin{proof}
Suppose $U \cap U' \neq \emptyset$ and $U \neq U'$.  
We will show that, up to switching $\cF^s$ and $\cF^u$, 
there exists a unique lozenge $L$ in U such that the $\cF^s$ saturation of $L$ does not contain $U$ but does contain $U'$ and a
unique lozenge $L'$ of $U'$ such that the $\cF^u$ saturation of $L'$ does not contain $U$' but does contain $U$.  
In this case, we will then have $U \cap U' = L \cap L'$. 

First, observe neither region can be properly contained in the other, so some leaf of $\partial U$ intersects (and therefore necessarily crosses) $U'$.  Up to switching $\cF^s$ and $\cF^u$, say this is a leaf $l^s$ of $\cF^s$.   
Notice that $l^s$ is entirely contained in $\partial U$, so must cross the scalloped region $U'$.  Thus, $l^s$ either contains the side of a lozenge of $U'$, or it passes through a lozenge of $U'$.  In either case, since $l^s$ is in the boundary of $U$, the product structure of $U$ 
%There is a unique lozenge $L'$ of $U'$ such that either 
%$\l^s \cap U'$ is contained in $U'$ or $l^s \cap U'$ is
%contained in $\partial L'$ and $L'$ intersects $U$.
%\sidenote{S: Not quite correct to say:
%``Then $l^s \cap U'$ lies in a single closed lozenge of $U'$". T:mmh, I'm confused. the sentence above makes no sense now: $\l^s \cap U'$ is contained in $U'$ is always true. I'm too tired to fix it now}   
%the product structure of $U$ 
then implies that $U$ lies in the $\cF^s$ saturation of a uniquely determined lozenge $L'$ of $U'$. This lozenge
$L'$ is  either crossed by $l^s$,  or $L'$ has
a side contained in $l^s$ and $L'$ intersects $U$.
This implies in particular that $U \cap U' = U \cap L'$.
Similarly, the product structure of $U'$ implies that $U'$ must lie in the $\cF^u$ saturation of a unique, well defined lozenge $L$ of $U$.  
Thus, $U \cap U' = U' \cap L$, and as a consequence
$U \cap U' = L \cap L'$, where $L, L'$ are well determined, unique
lozenges, $L$ of $U$ and $L'$ of $U'$.

Note that it is possible that $L = L'$, and the lozenges are equal to $U \cap U'$. In that case $U \cup U'$ are part of a chain of lozenges.
% or that the intersection $U \cap U'$ is properly contained in $L \cap L'$. 
%\sidenote{S: This is the broken sentence (starts with "or that the intersection..". This should  be removed
%since we should not have the proper containment.}
%\sidenote{S: I got confused here, I got that $U \cap U'$ is never
%properly contained in $L \cap L'$. Am I wrong? T: agreed. I think the sentence was probably  meant to ne ''or $L\cap L'$ is properly contained in both $L$ and $L'$ or something like that.
%S: The sentence stating the proper containment has been removed.}

To prove the second statement in the lemma, assume again for concreteness that $U$ lies in the $\cF^s$ saturation of $L'$. Here $L'$ is
a well defined, unique lozenge of $\alpha(U) = U'$.  
Let $I$ be the interval of leaves of $\cF^s$ that meet the closure $L$. 
This is a proper subset of the interval of leaves that meet $U$.  Thus, $\alpha(I)$ is a subset of the interior of $I$, so $\alpha$ has a fixed leaf that
intersects $L'$.  Similarly, applying $\alpha^{-1}$ to the interval of leaves of $\cF^u$ meeting the closure of $L'$, we find a fixed leaf for $\alpha$ that
intersects $L$. Hence there is a  fixed point for $\alpha$ on $\orb_\phi$, which lies in $L \cap L'$.  %Since this fixed point of $\alpha$ lies in %\sidenote{S: Humm... ``the interior of $L \cap L'$" this set
%is open...} 
%$L \cap L'$.   which is a subset of the scalloped regions.
Now it is easy to see that this fixed point of $\alpha$ cannot be the corner of a lozenge because it lies in a scalloped region (see, e.g., \cite[Lemma 2.29]{BFM}), so $\alpha$ has only
one fixed point 
by Proposition \ref{prop:fixed_chain}. Moreover, $\alpha (L) = L'$, so $L\neq L'$.
\end{proof}  
%\sidenote{S: In the end, COOL LEMMA!}

%Let $\cA^s$ (resp. $\cA^u$)  be the set of stable 
%(resp. unstable) leaves
%intersecting $U$. Both of these sets are homeomorphic to $\mathbb{R}$.
%We know that $\alpha(U), U$ intersect. Since $U$ is a scalloped
%region then $\alpha(\cA^s), \cA^s$ are comparable with each
%other and so are $\alpha(\cA^u), \cA^u$.
%So either I) $\alpha(U) = U$ and $\alpha(\cA^u) = \cA^u$;
%or II) $\alpha(\cA^u)$ is a strict subset of $\cA^u$
%(implying $\cA^s$ a strict subset of $\alpha(\cA^s)$);
%or III) $\cA^u$ is a strict subset of $\alpha(\cA^u)$
%(implying $\alpha(\cA^s)$ is a strict subset of $\cA^s$).
%We remark that 
%$\alpha(U) \not = U$ if and only if either II) or III) occurs.
%
%In situations II) and III) the same happens for positive powers
%of $\alpha$ and there is a single point $p$ in $U$ which
%is fixed under $\alpha$.

\begin{definition}\label{def_tree_scalloped}
A {\em tree of scalloped regions} $T \subset \orb_\phi$ is a chain of lozenges such that each lozenge in $T$ shares each of its four sides with some other lozenge in $T$.
\end{definition}

The terminology ``tree of scalloped regions" is justified by the following elementary observation. 
\begin{lemma} \label{lem:tree_scalloped} 
Let $T$ be a tree of scalloped regions.  Then any lozenge $L$ in $T$ is contained in exactly two distinct scalloped regions, each a subset of $T$.  
\end{lemma} 

\begin{proof} 
Pick a pair of opposite sides of $L$.  By definition of tree of scalloped regions, both sides are adjacent to other lozenges of $T$.  Iterating this, one produces a bi-infinite line of lozenges in $T$.  By \cite[Theorem 5.2]{Fenley_structure_branching}, the union of these lozenges gives a scalloped region.  One also may produce a line using the other pair of opposite sides, giving a distinct scalloped region, these are the unique scalloped regions containing $L$. 
\end{proof}

We note one more fact for future reference. 

\begin{lemma} \label{lem:tree_fix_corners}
If $T$ is a tree of scalloped regions in $\orb_\phi$, then there exists a nontrivial $g \in \pi_1(M)$ fixing all corners of lozenges in $T$.   Furthermore, $T$ is the maximal chain of $g$-invariant lozenges: Any chain preserved by $g$ is a sub-chain of $T$.  
\end{lemma} 

\begin{proof} 
The first statement is direct consequence of Proposition \ref{prop:nonsep_implies_fixed}.  For the second statement, note first that $T$ is a maximal chain, so by Remark \ref{rem_maximal_chain} no other lozenge has corners fixed by $g$ (or any of its powers).  Thus, if $\cC$ were another invariant chain, then $g$ would permute the lozenges of $\cC$ and thus act without fixed points in $P$, which is a contradiction.  
\end{proof}

%-%%%%%%%%%%%%%%%%%
\subsection{Cutting surfaces in $M$} \label{subsec_cutting_surfaces}
We now describe topological features of a 3-manifold $M$ with a pseudo-Anosov flow that give rise to the features in the orbit space of the flow described above.  The case of interest to us is when $M$ has nontrivial JSJ decomposition.   We do not assume $M$ is orientable, so the cutting surfaces of the JSJ decomposition may be tori or Klein bottles.  When $P$ is a Seifert piece, the {\em boundary} of $P$ is the union of cutting surfaces adjacent to $P$.  Note that this is not necessarily the same as the topological boundary of $P$ in $M$ as various boundary surfaces of $P$ may be glued to each other in pairs.  We recall two definitions briefly mentioned in the introduction. 

\begin{definition}[Following \cite{BF_totally_per}]
 $P$ is called a \emph{periodic Seifert piece} for $\phi$ if there is a Seifert fibration of $P$ such that a regular fiber is freely homotopic to a periodic orbit of $\phi$.   
\end{definition} 

\begin{definition}\label{def:scalloped}
A $\pi_1$-injective torus or Klein bottle in $M$ is called a {\em scalloped surface} if its fundamental group contains two independent elements that are freely homotopic to periodic orbits of $\phi$.  A periodic Seifert piece $P$ is called a {\em scalloped piece} if {\em each} of its boundary surfaces is a scalloped surface.  
\end{definition} 

\begin{rem} \label{rem:PneqM}
By \cite{Barbot96} (and \cite[Theorem 4.1]{BarbFen_pA_toroidal} for the pseudo-Anosov case), 
 a pseudo-Anosov flow on a closed Seifert manifold is necessarily orbit equivalent to the lift of a geodesic flow on the unit tangent bundle of a hyperbolic orbifold (and hence Anosov).  In particular, a closed Seifert manifold cannot be a periodic Seifert piece for such flow, nor can a closed Seifert manifold admit a scalloped surface.  
\end{rem}

\begin{convention*}
In this article we will say that a set $\wt S$ in $\wt M$ \emph{projects} to a chain of lozenges $\cC$, or a scalloped region $U$, in $\orb_\phi$ if the closure of $\pi(\tilde S)$ is equal to the closure of $\cC$, or the closure of $U$, in $\orb_\phi$, where $\pi\colon \wt M \to \orb_\phi$ denotes the projection map.
\end{convention*}
The reason for the above convention is that we will be considering lifts $\wt S$ of tori or Klein bottles, such that their projections $\pi(\wt S)$ always contain the interiors of the lozenges in a chain $\cC$, but may either contain a corner or a side of each lozenge (though never both).  What is of primary importance to us is that these chains or scalloped regions are naturally associated to $\wt S$ and invariant under the appropriate conjugate of $\pi_1(S)$ when $\wt S$ is the lift of $S$ to $\wt M$. 

The following statement is shown within the proof of \cite[Theorem B]{Barbot_MPOT} (but with different terminology there). See also Lemma 5.5 of \cite{BarbFen_pA_toroidal} for a proof of a more general statement.  While these assume $M$ orientable (and hence take $\pi_1(S) = \mathbb{Z}^2$), orientability is not used in an essential way, and one obtains the result below with the same proof.  

\begin{proposition}[ \cite{Barbot_MPOT,BarbFen_pA_toroidal}]  \label{prop:scalloped_projection}
Let $S$ be a scalloped surface and $\wt{S}$ a lift of $S$ to $\wt{M}$.  Then, up to homotopy of $S$, the projection 
of $\wt{S}$ to $\orb_\phi$ is a scalloped region. 

Moreover, if $\wt{S}$ is chosen to be the lift invariant under $\pi_1(S)$, then there are independent elements $g_1$ and $g_2$ for $\pi_1(S)$ such that $g_1$ translates one family of lozenges making up the scalloped region, while preserving each lozenge in the other, and $g_2$ preserves the lozenges of the first family and translates (or, in the case of a Klein bottle, translates and reflects) the other.  
\end{proposition} 

Our next goal (Lemma \ref{lem:scalloped_tori_are_JSJ}) is to show that embedded scalloped surfaces are always isotopic to 
cutting surfaces of the JSJ decomposition of $M$.  Before this, we recall a few results from \cite{BarbFen_pA_toroidal} that show cutting surfaces of
 the JSJ decomposition can be taken in good position with respect to the flow.     

\begin{definition}
A {\em Birkhoff annulus} is an immersed closed annulus in $M$ bounded by periodic orbits of $\phi$, with interior transverse to $\phi$.
A {\em Birkhoff surface} is a torus or Klein bottle formed by a union of Birkhoff annuli glued along their boundary components.  
\end{definition} 
 
\begin{definition}
A Birkhoff surface or annulus $S$ is called {\em weakly embedded} if the complement of the periodic orbits in $S$ is embedded.  
\end{definition}

The following result relates Birkhoff surfaces and invariant chains of lozenges. The forward direction is immediate from the definition, and the backward direction is proven in Proposition 6.7 of \cite{BarbFen_pA_toroidal}.\footnote{There is a typo in the statement of Proposition 6.7 of \cite{BarbFen_pA_toroidal}, where they write \emph{string} of lozenges, instead of chain. But the proposition and its proof holds for any chain of lozenges, not just strings.}
\begin{proposition}[\cite{BarbFen_pA_toroidal} Proposition 6.7] \label{lem:Z2_invariant_chains} 
Let $S$ be a Birkhoff surface.  Then a lift $\wt{S}$ to $\wt{M}$ projects to a chain of lozenges invariant under (a conjugate of) $\pi_1(S)$. 

Conversely, suppose $\cC$ is a chain of lozenges that is invariant by a subgroup $\Gamma$ of $\pi_1(M)$, isomorphic to $\bZ^2$ or $\pi_1(K)$, with no strict subchain preserved by $\Gamma$.   
Then there exists a $\pi_1$-immersed torus or Klein bottle in $M$, with fundamental group (conjugated to) $\Gamma$ and such that one of its lifts projects to $\cC$.
\end{proposition}

Using this, we can show the following converse statement to Proposition \ref{prop:scalloped_projection}.  

\begin{lemma}\label{lem_scalloped_region_to_torus}
Let $U$ be a scalloped region in $\orb_\phi$.  Then there is a $\pi_1$-immersed torus or Klein bottle $S$ in $M$ with a lift $\wt{S}$ in $\wt{M}$ which projects to $U$.  Moreover, the stabilizer of $U$ contains the corresponding conjugate of $\pi_1(S)$ as a finite index subgroup. 
\end{lemma} 

\begin{proof} The first part of the argument follows \cite[Lemma 2.33]{BFM}, we outline it for convenience.  
Let $U$ be a scalloped region, and consider the finite index subgroup of its stabilizer that preserves each of the four families of boundary leaves. 
Considering the action on the families  $l^{1,s}_k$, and $l^{1,u}_k$, each $g$ sends $l^{1,s}_k$ to some $l^{1,s}_{k+m}$ (where $m$ depends only on $g$).  This gives a homomorphism from this group to $\bZ^2$, whose image is a rank 2 subgroup (so isomorphic to $\bZ^2$) since Proposition \ref{prop:nonsep_implies_fixed} shows that some nontrivial elements act trivially on each factor.  
Now apply Proposition \ref{lem:Z2_invariant_chains}, to find the Birkhoff torus or Klein bottle $S$.
\end{proof}

The next lemma says that, after isotopy, one may upgrade such a scalloped surface to be transverse to the flow.  

\begin{lemma}  \label{lem:scalloped_isotopic}
Suppose that $S$ is a scalloped surface in $M$.  Then $S$ is isotopic to 
a surface transverse to the flow.   This surface is unique up to isotopy along flowlines.  
\end{lemma} 
This is proved in \cite{Barbot_MPOT} when $M$ is orientable and contains no one-sided Klein bottles.  We give a proof that applies in the general setting. 

\begin{proof} 
Let $S$ be a scalloped surface.  
By Proposition \ref{prop:scalloped_projection}, up to a homotopy
of $S$, there is a lift $\wt{S}$ of $S$ to $\wt{M}$ that projects to a scalloped region in $\orb_\phi$.  Let $U \subset \wt{M}$ be the union of orbits in this scalloped region; this is fixed by $\pi_1(S)$ and the quotient of $U$ by $\pi_1(S)$ is homotopic to $S$.   
Let $G$ be the stabilizer of $U$ in $\pi_1(M)$.  Then (with appropriate choice of basepoint) $\pi_1(S) \subset G$, and is a subgroup of finite index by Lemma \ref{lem_scalloped_region_to_torus}.  
We will show that in fact $G = \pi_1(S)$, by using some 3-manifold topology.  
%First we show that $\pi_1(S)$ has finite index in $G$: let $G'$ be the
%subgroup of $G$ which preserves transverse orientations to the weak
%stable and weak unstable foliations. Then $G'$ is finite index in
%$G$ and $G'$ is isomorphic to $\mathbb{Z}^2$. There is a corresponding
%finite index subgroup of $A$ of $\pi_1(S)$ and again $A \cong {\mathbb{Z}}^2$.
%Obviously $A$ has finite index in $G'$ which implies that $\pi_1(S)$
%has finite index in $G$.
%\textcolor{blue}{Is it immediate that pi1(S) is finite index?  this follows easily from Barbot using the orientable cover, or by arguing about the possible stabilizers of a scalloped region, but we should say something...?}
 
Let $M_1$ be the (infinite) cover of $M$ associated with $G$.  
Let $\psi$ be the flow lifted to $M_1$ and let $V$ be the projection of $U \subset \wt M$ to $M_1$. When restricted
to $V$, the flow $\psi$ is a product flow with
orbit space a compact surface $C'$ with fundamental group
isomorphic to $G$.
In other words, $V$ is homeomorphic to $C' \times \R$, with flowlines of $\psi$ giving the product decomposition, 
$\pi_1(C') = G$, and there is an embedded copy of $C'$ transverse to the flow. 
Since $\pi_1(S) \subset G$ with finite index, we have that $S$ lifts to an embedded compact surface 
$S’$ in $M_1$. 
The surface $S'$ is homotopic in $M_1$ to an embedded surface in $V$ (which carries all the topology
of $M_1$). Since $M_1$ is irreducible and Haken, $S'$  is isotopic into $V \cong C' \times \R$, by a classical result of Waldhausen \cite{Waldhausen}. We further get that $S'$ is isotopic into $C' \times [-1,1]$.
Theorem 10.5 of \cite{Hempel} now implies that $S’$ is  isotopic to $C’$, which shows that $\pi_1(S)$ is in fact equal to $G$.

\vskip .05in
 The next step is to show that the projection of $C'$ to $M$ can be homotoped to an embedded surface.  
Denote the projection of $C'$ by $C$. Since $\pi_1(S) = G$, $C$ is homotopic to $S$ in $M$. 
Put $C$ in general position moving points slightly along flow lines.  If it is already embedded, we are done; otherwise
consider the
finitely many curves of self intersection of $C$. 
As in the proof of \cite[Theorem B]{Barbot_MPOT}, one
may apply an isotopy to remove any null-homotopic curves of self-intersection, by sliding part of
the immersed surface $C$ while keeping the rest fixed.
This is done carefully in lemmas 7.11, 7.12 and 7.13, pages
144--147 of \cite{Barbot_MPOT}.
%\sidenote{this seems vague, maybe better to just refer to the right paragraph where this is done?}

Now we show that $C$ cannot have any non null-homotopic curves of self-intersection.  Assume for contradiction that
 $\tau$  is such a curve in $C$.  
 Think of $C$ as the image of an immersion
$f\colon L \to M$ where $L$ is a closed surface, and
let $\tau^* \subset L$ be a 
simple closed curve in $L$ with $f(\tau^*) = \tau$.   
We will use this to produce a curve $\alpha \in \pi_1(M) \setminus \pi_1(S)$, with the property 
that $\alpha(U) = U$. This contradicts the fact proven earlier that $\pi_1(S)$ is the stabilizer of $U$.   

A priori, the restrictions of the map $f$ to $\tau^*$ may or may not be injective, and we treat these cases separately.
 
\begin{case}[$f\colon \tau^* \to \tau$ is not injective] 
Since $f$ is locally injective, we may then find a sub-arc of $\tau^*$
where $f$ is injective on the interior but whose endpoints are mapped to the same point in $M$. 
Thus the image is a 
%Recall that $f$ is locally injective. Start at a point in $\tau^*$ 
%and move injectively along $\tau^*$ until the first point
%with same $\tau^*$ image as the initial point.
%The image under $f$ is a 
closed curve $\alpha$ in $M$, contained in $C$. Considering
$\alpha$ as a deck transformation, we have $\alpha(U) \cap U 
\not = \emptyset$.  By Lemma \ref{lem:intersection_scalloped}, 
either $\alpha(U) = U$, or $\alpha$ has a unique fixed point in $U$.  
Since $f\colon \tau^* \to \tau$ is a finite cover, it follows
that some positive power $\alpha^n$ is the image of an 
element in $\pi_1(L)$, i.e., in $\pi_1(S)$.  This means that $\alpha^n(U) = U$, which is 
impossible since $\alpha$ has a unique fixed point inside $U$. 
\end{case} 
%----------------------------------------
  
\begin{case}[$f\colon \tau^* \to \tau$ is injective] 
We now assume that for \emph{any} self intersection
curve $\tau$ of $C$ and preimage $\tau^*$ under $f$, then the restriction 
$f\colon \tau^* \to \tau$ is injective.
Since $\tau$ is a curve of self intersection of $C$, it
follows that $f^{-1}(\tau)$ contains another component $\beta^*$
besides $\tau^*$. There is an arc $\alpha^*$ in $L$ from
$\tau^*$ to $\beta^*$. We can adjust the endpoint in $\beta^*$ so
that $\alpha := f(\alpha^*)$ is a closed curve in $M$.
As in the first case, $\alpha(U) \cap U \neq \emptyset$.

There are three possibilities in this situation: 
\begin{enumerate}
\item  The curves $\tau^*, \beta^*$ are both two sided in $L$. Then $\tau^*, \beta^*$ are isotopic in $L$, and
applying $f$ gives $\alpha^{-1} \tau \alpha = 
\tau^{\pm 1}$
\item The curves $\tau^*, \beta^*$ are both one sided in $L$ (so $L$ is necessarily a Klein bottle). Then the squares of these curves are homotopic to
curves in $L$ which are isotopic to each other. Consequently, 
$\alpha^{-1} \tau^2 \alpha
= \tau^{\pm 2}$; 
\item Up to switching $\tau^*, \beta^*$, we have that $\tau^*$ is
one sided and $\beta^*$ is two sided in $L$. 
Consequently, $\alpha^{-1} \tau^2 \alpha
= \tau^{\pm 1}$. 
\end{enumerate} 

Since $\tau(U) = U$, the equations in possibilities (1)--(3) above imply that  either $\tau \alpha(U) = \alpha(U)$ or 
$\tau^2 \alpha(U) = \alpha(U)$. Thus we always have $\tau^2 \alpha(U) = \alpha(U)$.
Suppose now that $\alpha(U) \neq U$. 
By Lemma~\ref{lem:intersection_scalloped}, there exist lozenges $L, L'$ of $U$ and $\alpha(U)$ respectively, 
such that $\alpha(U) \cap U = L \cap L'$ and $L\neq L'$.
Thus, $\tau^2$ fixes both $U$ and $\alpha(U)$, and it therefore  
fixes the four points of intersection of the sides of $L$ and $L'$. This is a contradiction since $\tau$ is nontrivial. 

Once again, we conclude that $\alpha(U) = U$, so $\alpha$ is
in $\pi_1(S)$ and does not produce self intersections of $C$.
\end{case} 
This shows that $C$ is embedded. As $C$ is embedded and 
homotopic to $S$, applying Waldhausen's theorem 
\cite{Waldhausen}, gives that $S$ is isotopic to $C$ as desired.

The proof of uniqueness of the surface $C$ follows exactly as in the argument in 
\cite{Barbot_MPOT} 
%\sidenote{? specific theorem / page reference?}
%it uses only the fact that $\wt{C}$ projects to 
%a scalloped region. 
\end{proof}

Building on Lemma \ref{lem:scalloped_isotopic}, the next proposition says that we have good representatives for both scalloped and non-scalloped cutting surfaces.  

\begin{proposition}[Barbot, Fenley]\label{prop:good_representative} 
Let $S$ be a cutting  surface of the JSJ decomposition of $M$. Then $S$ is homotopic to a weakly embedded Birkhoff surface, and 
\begin{enumerate} 
\item  If $S$ is scalloped, then it may be isotoped to a surface transverse to the flow, which is unique up to isotopy along flow lines.  
\item  If $S$ is not scalloped then the Birkhoff surface is not
unique up to homotopy along flow lines, %\sidenote{K: i deleted the long margin conversation because i think this is still ok as written}
but $\pi_1(S)$ leaves invariant a unique minimal bi-infinite chain of lozenges $\cC$, and for any Birkhoff surfaces $S$, the unique lift $\wt{S}$ stabilized by $\pi_1(S)$ projects to $\cC$. %\sidenote{T: How do you like it this way? it makes use of our convention}
% \sidenote{S: The statement about chain of lozenges is correct.
%But the next about orbits is not correct and should be 
%removed. In fact this is remarked in the paragraph after the
%convention, where we say the projection may contain the corner
%orbit or a side but not both, so set of orbits is not well defined.
%If wanting to say smthg along these lines it would be that the
%closure of the set of orbits is well defined.}
%and
% the set of orbits of $\widetilde{\phi}$ which intersects
%the unique lift $\wt{S}$ stabilized by $\pi_1(S)$ 
%is well defined. 
\end{enumerate} 
We call the transverse surface (in the first case), or the weakly embedded Birkhoff surface (in the second) a {\em good representative} for the cutting surface. 
\end{proposition} 

Given Lemma \ref{lem:scalloped_isotopic}, it remains only to prove statement (2).  This is proved in \cite{BarbFen_pA_toroidal}, but is not completely transparent since it occurs within the proof of Theorem 6.10.   We give an outline and further references, and provide the necessary details to treat the possibly non-orientable setting.  

\begin{proof}
Item (1) has already been proved in Lemma \ref{lem:scalloped_isotopic}.  So, we suppose that $S$ is not scalloped.   
%First, if $M$ is non-orientable, notice that one can prove the above lemma in the orientable cover of $M$, and then project back down to $M$ to obtain the desired conclusion, except for isotopy in item (1). We will prove that
%at the end.. So we may assume that $M$ is orientable.
We assume first that $M$ is orientable.  
Let $S$ be a cutting surface of the JSJ decomposition of $M$, since we are in the orientable setting $\pi_1(S) = \bZ^2$ and therefore by \cite[Lemma 5.3]{BarbFen_pA_toroidal} it preserves a chain of lozenges.  (If there are two distinct such chains, then \cite[Lemma 5.5]{BarbFen_pA_toroidal} states that $S$ must be scalloped.)

Let $\cC$ be a minimal (with respect to inclusion) $\pi_1(S)$-invariant chain of lozenges. 
Since $S$ is an embedded torus, Step 1 of the proof of Theorem 6.10 of \cite{BarbFen_pA_toroidal} shows that no corner of $\cC$ has any point of its orbit under $\pi_1(M)$ in the interior of any lozenge of $\cC$.  In the language of \cite{BarbFen_pA_toroidal}, $\cC$ is \emph{simple} (see also \cite[Corollary 6.11]{BarbFen_pA_toroidal}).  Thus by \cite[Proposition 6.7]{BarbFen_pA_toroidal}, $\cC$ corresponds to a weakly embedded Birkhoff torus that is homotopic to $S$. More precisely $\cC$ is the projection to the orbit space of the unique lift $\wt{S}$ of $S$ stabilized by $\pi_1(S)$.    This proves item (2) in the case where $M$ is orientable.  

%When $S$ is scalloped, Théorème B of \cite{Barbot_MPOT} gives us that the torus $S$ is in fact isotopic to a transverse torus, unique up to isotopy along flowlines; proving item (1).  In this case, the $\pi_1(S)$-invariant chain $\cC$ is by definition a scalloped region, so can be realized in two ways as a line of lozenges.  This is the reason uniqueness of $\cC$ fails in this case.  

If $M$ is non-orientable,  one may pass to the orientable double cover $\hat M$ of $M$.  A cutting surface $S$ in $M$ will lift to an incompressible torus $\hat{S}$ in $\hat{M}$. We only need to justify that $\hat S$ is also a cutting surface of $\hat M$. This is by definition if at least one side of $\hat S$ is atoroidal.
%\sidenote{T: I added this just because it seemed weird not to mention that possiblity at all, even though it is obvious,
%S: Obvious yes, but enlightening, I vote to keep this.} 
If both sides of it are Seifert fibered then these project down to necessarily distinct Seifert fibered pieces of $M$ (since $S$ is a cutting surface), so the fibers on either sides are different homotopy classes in the torus; hence $\hat{S}$ is a cutting surface for $\hat M$ and we may apply the argument above to find a good representative in $\hat {M}$ and then project back down to $M$.
\end{proof}

%\begin{remark} 
%When $S$ is scalloped, there is still a minimal $\pi_1(S)$-invariant chain $\cC$;\sidenote{T: this remark is confusingly written. Should we rewrite it a cut it off completely? I don't know that it brings much,
%S: I am fine removing it. Does not bring much. I was also confused
%at first because it mentions chains for scalloped then strings
%which are orthogonal.} however, in some cases $S$ is not unique.  Specifically, if $\cC$ is not a {\em string}, one can build several distinct homotopic Birkhoff surfaces (which can only differ at certain corners of the chain, see Corollaire 5.6 of \cite{Barbot_MPOT}). 
%\end{remark} 

Good representatives for cutting surfaces can be obtained simultaneously, as follows. 
\begin{proposition}[Good form for cutting surfaces \cite{Barbot_MPOT}, Th\'eor\`eme E] \label{prop:good_form}
Given any number of distinct cutting surfaces, one can choose good representatives that are either disjoint, or have pairwise intersections only along periodic orbits of Birkhoff surfaces.  Moreover, when two such Birkhoff surfaces intersect, they do not cross each other. 
\end{proposition} 
While Barbot proves this in the setting where $M$ is orientable, the proof carries through in general, given our Proposition \ref{prop:good_representative} above.

Using this, we make the following definition which we will use frequently in the next section.  

\begin{definition} 
Suppose $P$ is a piece of the JSJ decomposition of $M$.  By Proposition \ref{prop:good_form} one may homotope each boundary surface to a good representative.  We call the resulting set (obtained from $P$ by this operation on boundary components) the {\em good representative} for $P$.  
\end{definition} 
Note that the good representative may not be a submanifold, because certain periodic orbits in boundary components may end up coinciding after homotopy -- this is also why homotopy rather than isotopy is needed in the definition.  

We record for future reference the following useful property of good representatives. 
\begin{lemma}[See Theorem B of \cite{BarbFen_pA_toroidal}] \label{lem:orbits_in_P}
If $P$ is the good representative of a periodic Seifert piece, then the only orbits contained in $P$ are periodic and have powers agreeing with a regular fiber.   
\end{lemma}
This is proved in Section 7 of \cite{BarbFen_pA_toroidal} for orientable manifolds, but holds generally by passing to the orientable cover.

\begin{lemma}\label{lem:scalloped_tori_are_JSJ}
Any embedded scalloped surface is homotopic to a 
cutting surface of the JSJ decomposition of $M$.
\end{lemma}

\begin{proof}[Proof of Lemma \ref{lem:scalloped_tori_are_JSJ}]
Let $S$ be a scalloped surface for a pseudo-Anosov flow $\phi$ in $M$. As explained in Remark \ref{rem:PneqM}, this implies that $M$ has nontrivial JSJ decomposition.  

Since $S$ is $\pi_1$-injective, it can be isotoped to be disjoint from all cutting sufaces, so lies in some (necessarily Seifert) piece $P$.    Our goal is to show $S$ is peripheral in $P$.   Since $P$ has nonempty boundary, there are no horizontal tori or Klein bottles, so one element of $\pi_1(S)$ represents the fiber direction.  Call this $g$.  

Consider the lift $\wt S$ of $S$ to $\wt M$ which is invariant under $\pi_1(S)$, and 
let $U$ be the projection of $\wt S$ to the orbit space, i.e., the scalloped region associated to $S$. 
The group $\pi_1(S)$ stabilizes $U$, and in fact is either equal to, or finite index in, the stabilizer of $U$ (Lemma \ref{lem_scalloped_region_to_torus}).  Now choose a boundary component $Z$ of $P$.  If $S = Z$ we are done; otherwise $\pi_1(Z)$ does not virtually stabilize $U$, so contains some element $f$ such that $f^2(U) \neq U$.  

However, $f^2 g f^{-2} = g$, which does stabilize $U$, and also stabilizes $f^2(U) \neq U$.  
Thus, $g$ stabilizes two {\em distinct} scalloped regions.   We claim this means that $g$ cannot act freely on $\partial U$.  If it did act freely, then it would freely permute the connected components of the complement of $U$.  However, the structure of scalloped regions forces $f^2(U)$ to intersect at least one, and at most two, connected components of $\orb \smallsetminus U$.  Thus, $g$ cannot act freely, and hence represents a periodic orbit of the flow.  In particular, $P$ is a periodic piece.  

Since $S$ is scalloped, by definition there exists an independent element $h\in \pi_1(S)$ representing a periodic orbit of $\phi$.  Recall that $U$ can be realized in two ways as a line of lozenges, and by Proposition \ref{prop:scalloped_projection}, we may choose such an $h$ such that it %take $g$ and $h$ such that the action of $g$ on $U$ 
fixes each lozenge in one line, and translates by a shift along the other.  

Let $\gamma$ be a periodic orbit of $\phi$ which is the
projection of 
one of the orbits fixed by $h$ in the boundary of $U$.  
Choose a good representative for the piece $P$ (abusing notation, we denote this also by $P$) so that all the periodic orbits in $P$ have a power freely homotopic to the fiber direction, as in Lemma \ref{lem:orbits_in_P}. Then $\gamma$ cannot be contained in $P$. 
If $\gamma$ is disjoint from $P$, then the free homotopy between $\gamma$ and a loop representing $h$ in $S \subset P$ must pass through a boundary of $P$. That is, $h$ is peripheral in $P$.  This implies that $S$ was in face a cutting surface, so we are done.  

Thus, we are left with examining the case where $\gamma$ intersects some boundary component $Z'$, but is not freely homotopic into it.  Since $P$ is assumed to be a good representative, by Proposition \ref{prop:good_representative}, $Z'$ is either scalloped and transverse or represented by a weakly embedded Birkhoff surface, thus $\gamma$ intersects it transversely.   
Let $\cC$ be the maximal $\pi_1(Z')$-invariant chain of lozenges, and $\wt \gamma$ the lift of $\gamma$ intersecting $\cC$. Since $\gamma$ intersects $Z'$ transversely and is assumed not homotopic into $Z'$, $\wt \gamma$ is not a corner but rather contained in the interior of a lozenge.  
But $h \in \pi_1(Z)$, so $\cC$ is $h$-invariant, meaning that $h$ cannot have any interior fixed point, a contradiction.  This eliminates this case and finishes the proof. 
\end{proof}

%----------
\subsection{Proof of Theorem \ref{thm_scalloped_tree_iff_scalloped_piece}}\label{sec_scalloped_periodic}

Theorem \ref{thm_scalloped_tree_iff_scalloped_piece} was the statement that for Anosov flows, trees of scalloped regions correspond to scalloped periodic Seifert pieces.  We in fact prove a more precise statement in the more general setting of pseudo-Anosov flows, as follows.  

\begin{theorem} \label{thm:tree_iff_periodic} 
Let $\phi$ be a pseudo-Anosov flow on a closed 3-manifold $M$.  
The orbit space of $\phi$ contains a tree of scalloped regions if and only if $\phi$ has a scalloped periodic Seifert piece $P$.
Moreover, each tree of scalloped regions is the maximal chain of lozenges fixed by (a conjugate of) the fiber direction of a unique scalloped periodic piece.  
\end{theorem}

\begin{proof}
Let $T \subset \orb$ be a tree of scalloped regions. By Lemma \ref{lem:tree_fix_corners}, there exists a nontrivial $g \in \pi_1(M)$ that fixes every corner of every lozenge in $T$.  Choose a minimal such $g$, i.e., such that no root of $g$ also fixes each corner.  
The proof of \cite[Proposition 1.2]{BFM} shows that the centralizer $C(g)$ of $g$ in $\pi_1(M)$ is (conjugate to) the fundamental group of a Seifert piece $P$, with $g$ representing the Seifert fiber; in other words $P$ is periodic. We assume $P$ to be a good representative of the piece. 

To prove that $P$ is scalloped, we analyze its boundary surfaces.  
Let $\wt P$ be the lift of $P$ to $\wt M$ which is invariant 
under $\pi_1(P)$.
Let $S$ be a torus or Klein bottle 
in the boundary of $P$. 
By Proposition \ref{prop:good_representative}, $S$ is 
either a transverse torus or a Birkhoff surface containing some periodic orbits with a power freely homotopic to $g$.
Let $\wt S$ be a lift of $S$ to $\wt P$ invariant under $g$.

By Proposition \ref{prop:good_representative}, $\pi_1(S)$ has a minimal invariant chain of lozenges $\cC$ which is the projection of $\wt{S}$; if $S$ is scalloped this is the associated scalloped region which can be expressed as a chain in two ways;  and if not the chain $\cC$ is unique.   
Since $g$ leaves invariant $\cC$ and fixes all corners of $T$ (as in Lemma \ref{lem:tree_fix_corners}) 
 the chain $\cC$ is contained in $T$. 

Fix some lozenge $L$ in $\cC$.  We will show that orbits passing through $L$ go from one scalloped boundary component $S_1$ of $P$ to another, $S_2$.  Since by definition of $L$, these are orbits passing through $S$, we conclude $S$ is either $S_1$ or $S_2$ and thus is scalloped.  

To show this, we use the structure of $T$.   Lemma \ref{lem:tree_scalloped} says that $L$ is contained in exactly two distinct scalloped regions in $T$, say $U_1$ and $U_2$. By Lemma \ref{lem:scalloped_tori_are_JSJ} together with Proposition \ref{lem:Z2_invariant_chains}, each $U_i$ corresponds to the lift to $\wt M$ of a cutting surface. Let $S_1$ and $S_2$ be the cutting surfaces whose lifts project to $U_1$ and $U_2$ respectively. Since $U_1$ and $U_2$ are both fixed by $g$, the surfaces $S_1$ and $S_2$ contain loops that are freely homotopic to the Seifert fiber of $P$. Therefore they must be in the boundary of $P$, rather than in other pieces.
Now all orbits whose images are in the interior of $L$ intersect both $\wt S_1$ and $\wt S_2$ when crossing $\wt P$.  Thus, $\wt S$ must be equal to either $\wt S_1$ or $\wt S_2$, so $S$ is scalloped.

Since the choice of $S$ was arbitrary, 
we conclude that $P$ is a scalloped Seifert piece, as claimed.
This proves the direct statement of the proposition.

Conversely, suppose that the we have a periodic Seifert piece $P$ such that each of its boundary surfaces $S_1, \dots, S_n$ are scalloped. Let $\wt P$ be a lift of $P$ to $\wt M$, and $G \simeq \pi_1(P)$ be the subgroup fixing $\wt P$.
Let $\gamma$ be the lift of a periodic orbit of $\phi$ which
is contained  in $\wt P$, and let $g\in G$ be the element representing the fiber of $P$, so that $g\gamma = \gamma$. 
Let $\cC$ be the unique maximal chain of lozenges fixed by $g$ (see Remark \ref{rem_maximal_chain}). 
 We will show that $\cC$ is a tree of scalloped regions, i.e., that every corner in $\cC$ admits a lozenge in each of its $2p$ quadrants (recall, we are assuming the flow is pseudo-Anosov so may have singular orbits). 
We first show the following:

\begin{claim} Any corner of $\cC$ is an orbit contained in the interior of $\wP$.
\end{claim}

\begin{proof}
Let $x'$ be a corner of $\cC$, so it is fixed by $g$, and so by Proposition \ref{prop:fixed_chain} there is a chain $\cC'$ of lozenges connecting $x'$ to $\gamma$.
Suppose for contradiction that $x'$ is not contained in $\wP$.
Then there is a lozenge $L$ in $\cC'$,
so that one corner of $L$, call it $x$ is contained in $\wP$,
and the other corner, call it $y$ is not contained in $\wP$.
Notice that $y$ cannot intersect $\wP$ for otherwise it would pass through one of the boundaries of $\wP$, and hence be in the interior of a lozenge with a corner in the chain $\cC$, which is an impossible configuration for corner points.  

Let $Y$ be the component of $\partial \wP$ separating $x$ from
$y$. 
Notice that the projection of $Y$ to the orbit space
contains $L$. Since $\pi(Y)$ is a scalloped surface, it follows
that the projection of $Y$ to the orbit space is a scalloped
region $U$. 
Now, this scalloped region $U$ intersects either both the stable leaves of $x$ and $y$ or their unstable leaves. Thus $x$ and $y$ are both in the forward flow side (in the first case, because points in $Y$ would have forwards orbits accumulating to $x$ or $y$) or both in the backward flow side of $Y$ (in the second case). Hence, $Y$ cannot separate $x$ from $y$, contradicting the definition of $Y$.
\end{proof}

Using the claim above, we now show $\cC$ is a tree of scalloped regions. 
Let $c$ be any corner in $\cC$. By the above claim,
the orbit $c$ is contained in the interior of $\wP$.
We want to show that each quadrant of $c$ contains a lozenge of $\cC$ with $c$ as corner.  To do this, pick any quadrant and 
take another orbit $b$ in that quadrant, close enough (in $\wt M$) to $c$ so that $b$ intersects $\wt P$. Then, since $P$ is a periodic Seifert piece, $b$ must also intersect one of the planes bounding $\wt P$ (which are the lifts of the 
boundary surfaces $Z_i$). Since these surfaces are all scalloped, $b$, seen in $\orb$, must be a point of a scalloped region $U$ invariant by $g$. In particular, the quadrant of $c$ containing $b$ has a lozenge in $\cC$. This concludes the proof. 
\end{proof}

%%%%%%%%%%%%%%%%%%%%%%%
\section{Spines and periodic Seifert flips}\label{sec:spine}
To precisely state the definition of a {\em periodic Seifert flip} and give its construction we need to recall some framework developed in \cite{BarbFen_pA_toroidal}, and generalize it to the case of non-orientable manifolds.   This is the notion of the {\em spine} of a periodic Seifert piece. 

In  \cite[Theorem B]{BarbFen_pA_toroidal} Barbot and Fenley show that, under the assumption that $M$ is orientable, any periodic Seifert fibered piece for a pseudo-Anosov flow on $M$ has a {\em spine}.  A spine is a connected finite union of weakly embedded, elementary Birkhoff annuli, such that any sufficiently small neighborhood of the spine is a representative for the piece $P$. {\em Elementary} means that the restriction of the stable or unstable foliations to one annulus do not have compact leaves aside from the boundary periodic orbits.  This spine serves as a combinatorial model for the flow on that piece --- one needs only to keep track of the way the annuli are assembled together, as well as a choice of orientation for the periodic orbits (see \cite[Theorem D]{BF_totally_per}). 

Orientability is only used at the end of the proof given in \cite{BarbFen_pA_toroidal}, and not in an essential way; the only modification required is that one must allows for non-orientable Birkhoff surfaces.  Thus, the construction of spines also applies to our more general setting.  For completeness, and because we will reference aspects of the construction later on, we recall the steps given in  \cite{BarbFen_pA_toroidal} and comment on the use of orientability.  

\begin{theorem}[Generalizing \cite{BarbFen_pA_toroidal} Theorem B to the non-orientable setting] \label{thm:spines_exist}
Let $P$ be a periodic Seifert fibered piece for a pseudo-Anosov flow on $M$ (not assumed orientable).  Then there exists a connected, finite union of weakly embedded, elementary Birkhoff annuli in $M$ such that any sufficiently small neighborhood of this union is a representative for the piece $P$. 
This union of annuli is called a {\em spine}. 
\end{theorem} 

\begin{remark}
As will be clear from the proof, the spine is constructed from combinatorial data in the orbit space.  Thus, it is unique up to isotopy along flowlines.  For this reason, one often says {\em the} spine rather than {\em a spine}, but we will occasionally want to work with an explicit realization of weakly embedded elementary Birkhoff annuli and so refer to this as a spine.  
\end{remark}

\begin{proof}
Let $P$ be a periodic Seifert fibered piece of an Anosov or pseudo-Anosov flow $\phi$ on $M$.  By Remark \ref{rem:PneqM} we have $P \neq M$.    
Let $h \in \pi_1(M)$ represent a regular fiber of $P$ and let $\alpha \in \orb_\phi$ be a point fixed by $h$.  

Let $\mathcal{C}_\alpha \subset \orb_\phi$ be the maximal chain of lozenges containing $\alpha$ as a corner.   Let $\mathcal{T}$ be the graph (in fact, it is easy to see this graph is a tree) whose vertices are corners of lozenges in $\mathcal{C}_\alpha$ with edges between two corners of a single lozenge.  This graph naturally embeds in $\mathcal{C}_\alpha$, which gives it the structure of a {\em fat graph}.   The action of $\pi_1(P)$ on $\orb$ preserves $\mathcal{T}$.  Since $\pi_1(P)$ contains a $\mathbb{Z}^2$ subgroup, and stablizers of corners are cyclic, some element of $\pi_1(P)$ must act freely  on $\mathcal{T}$.  

The {\em pruned tree} $\mathcal{T}' \subset \mathcal{T}$ is defined to be the unions of all axes of elements of $\pi_1(P)$ acting freely (see \cite{BarbFen_pA_toroidal}).  If $P$ is assumed to be scalloped and the flow Anosov, one will always have $\mathcal{T}' = \mathcal{T}$, as we describe below.  In general, this may not be the case, and one instead proves that $\mathcal{T}'$ is $\pi_1(P)$-invariant and connected, hence a subtree (see \cite[p.~1935]{BarbFen_pA_toroidal}). 

Since the scalloped Anosov case is that of interest to us, we note some other properties special to this setting and explain briefly why $\mathcal{T}' = \mathcal{T}$.  In the scalloped, Anosov setting, $\mathcal{C}_\alpha$ will be a tree of scalloped regions and $\mathcal{T}$ is a 4-regular tree. 
All the elements whose powers are not a power of the fiber act freely on the graph and each vertex is in some axis (because for every vertex $v$ there are two $\bZ^2$ subgroups that leave invariant a scalloped region whose boundary contains $v$). In the pseudo-Anosov case, the same holds except that some vertices may have degree $2p$, with $p>2$, coming from the singular prongs.

The next step is to show that no element of $\pi_1(M)$ sends a vertex of $\mathcal{T}'$ to a point in the interior of a lozenge of $\mathcal{T'}$.  
Again, this is easy to show in the case where $\mathcal{C}_\alpha$ is a tree of scalloped regions (no point inside a scalloped region can be a corner of a lozenge, see e.g., \cite[Lemma 2.29]{BFM}), and in general this follows from arguments of \cite[Section 6]{BarbFen_pA_toroidal}, where the assumption that $M$ is orientable is not used.  
This property implies that each lozenge of $\mathcal{T}'$ corresponds to the lift of a {\em weakly embedded} Birkhoff annulus in $M$ -- it is the projection to $\mathcal{O}$ of the lift of a Birkhoff annulus (with boundary given by the corners of the lozenge), and by \cite[Theorem D]{Barbot_MPOT}, the fact that no vertex is sent to the interior of the lozenge by an element of $\pi_1(M)$ means this Birkhoff annulus can be taken weakly embedded. 
Moreover, \cite[Proposition 6.7]{BarbFen_pA_toroidal} (which also does not assume $M$ is orientable) shows that this may be done simultaneously for all lozenges of $\mathcal{T}'$ in a $\pi_1(P)$-equivariant way.

We now consider the quotient of $\mathcal{T}'$ by $\pi_1(P)$.  Since $\pi_1(P)$ is finitely generated, and $\mathcal{T}'$ is a union of axes of elements of $\pi_1(P)$, the quotient contains no infinite ray, and one then argues easily that $\mathcal{T}'/\pi_1(P)$ is finite.  Consider the (finite) union of lozenges making up a finite fundamental domain for this action, and the corresponding weakly embedded Birkhoff annuli.  Standard cut and paste arguments allow one to realize the {\em union} of these annuli as a weakly embedded surface -- this is in fact the same machinery as is used to obtain weakly embedded surfaces in the proof of Propositions \ref{prop:good_representative}. 
Let $B$ denote this union.  One now must show that some small neighborhood of $B$ in $M$ is a representative of the piece $P$, which completes the proof. 
 
Here the arguments of Section 7 of \cite{BarbFen_pA_toroidal} go through without modification.  In brief, one considers a small neighborhood $U$ of $B$, and lifts it to a connected subset $\wt U \subset \wt M$.  If $U$ is chosen small enough, since $\mathcal{T}'$ was $\pi_1(P)$-invariant, $U'$ will be diffeomorphic to $\Sigma' \times \mathbb{R}$, where $\Sigma'$ is a surface that deformation retracts to the tree $\cT'$ to give its fatgraph structure.  Thus, $\wt{U}$ is simply connected.  This implies that $U$ has incompressible boundary, hence is a Seifert piece representing $P$.   
\end{proof}

\begin{rem}
Our primary use of spine is in the setting where $P$ is a scalloped Seifert piece for an Anosov flow.  In this setting, the combinatorial data of the spine is encoded by the quotient of $\mathcal{T} = \cT'$ (the fat tree corresponding to the $\pi_1(P)$-invariant tree of scalloped regions in $\orb$) and its corresponding Birkhoff surfaces, which are a union of Birkhoff annuli, one for each lozenge of $\mathcal{T}/\pi_1(P)$.  In the next section we will reverse this process, showing that each periodic Seifert piece has nice coordinates by assembling it out of basic ``model" pieces.  This will allow us to describe the flip of a flow in coordinates.  
\end{rem}
 
\begin{definition}
Two flows $\psi$ and $\phi$ on $M$ with a common periodic Seifert piece $P$ are said to have {\em the same spine} in $P$ if there exists a union of weakly embedded annuli in $M$ that are elementary Birkhoff annuli for both flows and whose union serves as a spine for both flows.  
\end{definition} 

Note that while having the same spine forces a direct correspondence between periodic orbits on $P$ (these are precisely the boundaries of the annuli), it does not make any requirement on the direction of these orbits.   This brings us to the definition of flip.  

\begin{definition} \label{def:flip}
A pseudo-Anosov flow $\psi$ on $M$ is \emph{obtained by a periodic Seifert flip} of $\phi$ if, up to isotopy equivalence, 
there is a scalloped periodic Seifert piece $P$ for $\phi$ (and $\psi$) that is a good representative for both flows,  and such that $\phi$ and $\psi$ are isotopic on $M \smallsetminus P$ and have the same 
spine in $P$, but each periodic orbit of $\phi$ in this spine has the opposite orientation as the identical periodic orbit of $\psi$.  
\end{definition} 

It is not obvious from the definition that a flip on a specific piece exists or is unique (see Remark \ref{rem:not_unique}).  However, we will show the following in Sections \ref{sec:existence} and \ref{sec:uniqueness_flip} respectively. 
\begin{theorem}[Existence] \label{thm:existence}
If $\phi$ is a Anosov flow with a periodic Seifert piece $P$, then there exists a flow $\psi$ obtained from $\phi$ by a periodic Seifert flip.  
\end{theorem} 

\begin{theorem}[Uniqueness] \label{thm:flip_unique}
If $\psi$ and $\psi'$ are each obtained from a transitive Anosov flow $\phi$ by a flip on a periodic Seifert piece $P$, then $\psi$ and $\psi'$ are isotopy equivalent.
\end{theorem}  

These are both essential to the classification theorems stated in the introduction, and their proofs occupy the next few sections.  We begin by 
giving a precise coordinate description of periodic Seifert pieces of Anosov flows.

\section{Model flows and their flips}\label{sec_model_flow}

Theorem \ref{thm:spines_exist} states that a periodic Seifert piece can be realized as a small neighborhood of a union of weakly embedded Birkhoff annuli.  As described in \cite[section 8]{BarbFen_pA_toroidal}, there is a nice coordinate model for a small neighborhood of a Birkhoff annulus; this is the manifold $N$ homeomorphic to $ I \times I \times S^1$ with flow generated by the vector field $X^+$ (or $X^-$) described below. 
Thus, to describe a periodic Seifert piece, it suffices to describe how to glue copies of $N$ together.  Such gluings are discussed in \cite[Section 6]{BF_totally_per} for orientable manifolds, and also described implicitly in the discussion at the end of the proof of Theorem B of \cite{BarbFen_pA_toroidal}.  
Some additional subtlety arises in the non-orientable case, so we give an explicit, and slightly different version of the construction, adapted to our purposes.  

\subsection{Model Seifert pieces} \label{subsec:model}
Let $N= I \times I \times S^1$, where $I= [-\pi/2,\pi/2]$ and $S^1 = \R/\bZ$, with coordinates $(x,y,z)$. Fix some $\lambda\gg 1$, and define vector fields $X^{\pm}$ by\footnote{In \cite{BarbFen_pA_toroidal}, the $y$ and $z$ coordinates are reversed compared to what we have here.  For our purposes it will be convenient to have the $S^1$ fiber as the last coordinate.}
\[
\begin{cases}
 \dot{x} = 0 ,\\
 \dot y = \cos^2(x) + \sin^2(y)\sin^2(x)\\
 \dot z = \pm \lambda \sin (x) \cos (y). 
\end{cases}
\] 

\begin{figure} 
 \labellist 
  \small\hair 2pt
     \pinlabel $y$ at 49 70 
     \pinlabel $x$ at 220 15 
   \pinlabel $z$ at 20 200
 \endlabellist
\centerline{ \mbox{
\includegraphics[width=5.5cm]{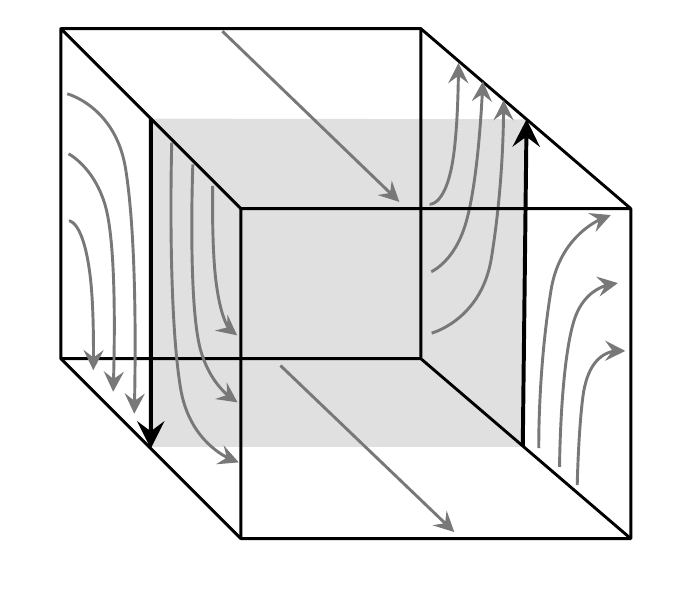}}}
\caption{The model block with flow $\psi^+$; The shaded region is a Birkhoff annulus.}
\label{fig:model_block}
\end{figure} 

Let $\psi^{\pm}$ be the flows defined by $X^{\pm}$.   Here and going forward, we abuse terminology somewhat and say ``flow" even though the orbits of $\psi^{\pm}$ enter and exit $N$, and so the flow is not defined for all time.  

We call $\psi^{-}$ the \emph{flip} of $\psi^{+}$, and vice versa.
Note that this flip is not the same thing as reversing the direction (i.e., reversing the time parameter of the flow), rather reversing time is obtained by a flip followed by a reflection in the $y$-coordinate.  
We emphasize some important properties of $\psi^{\pm}$

\begin{observation} \label{obs:properties_model_flip}
The flows $\psi^\pm$ have the following properties
\begin{itemize}
 \item There are only two closed orbits $\alpha_1 = \{-\pi/2\} \times \{0\} \times S^1$ and $\alpha_2 = \{\pi/2\} \times \{0\} \times S^1$, with $\alpha_1$ going in the positive $z$-direction for $\psi^{+}$ and negative $z$-direction for $\psi^{-}$ and the opposite for $\alpha_2$.
 
 \item Both flows are incoming through the boundary $I  \times \{-\pi/2\} \times S^1$ and both are outgoing through the boundary $I  \times \{\pi/2\} \times S^1$.
 \item Both flows are tangent to the boundaries $\{\pm \pi/2\} \times I  \times S^1$.
\item For both flows, we have that $\{-\pi/2\} \times [-\pi/2,0] \times S^1$ is the stable manifold of $\alpha_1$ and $\{-\pi/2\} \times [0,\pi/2] \times S^1$ is the unstable manifold of $\alpha_1$. Similarly, $\{\pi/2\} \times [-\pi/2,0] \times S^1$ and $\{\pi/2\} \times [0,\pi/2] \times S^1$ are (respectively) the stable and unstable manifolds of $\alpha_2$ in $N$ for both flows.
\item The annulus $B= [-\pi/2,\pi/2] \times \{0\} \times S^1$ is a Birkhoff annulus of both flows.
\end{itemize}
\end{observation} 

Notice also that the parameter $\lambda$ controls how much shearing there is in the $z$-direction as orbits go through the block, the larger the $\lambda$, the more shear. This control is used to prove hyperbolicity of the flows built from gluing such pieces, see Lemma \ref{lem_partial_gluing_periodic_piece}.  For the remainder of this section, we fix such a $\lambda$ and suppress any dependence of the flows $\psi^\pm$ on $\lambda$.  

It is easy to show that any embedded Birkhoff annulus in a flow on a manifold $M$ has a neighborhood that is orbit equivalent to $N$ with the model flow $\psi^+$ (or equivalently with the model flow $\psi^-$, as there is a homeomorphism of $N$ obtained by reflecting in the $x=0$ plane taking one to the other).  Annuli that are only weakly embedded admit neighborhoods that are local embeddings of $N$, injective on the interior but possibly non-injective on the faces containing the periodic orbits.

Two copies of $(N, \psi^+)$ can be glued together along a stable or unstable manifold face of a periodic orbit as in Figure \ref{fig:gluings} left, provided that they are oriented so that the incoming boundaries are adjacent, and the flows agree on the glued face.  This is achieved by flipping the $z$-coordinate direction of one of the blocks.   Gluing four around a corner produces a standard neighborhood of a periodic orbit in a periodic Seifert piece; using only two pieces produces singular Seifert fibers.   This is described explicitly in coordinates in the following constructions.  After describing these local models, we will then prove that every periodic Seifert piece of an Anosov flow is obtained by gluings of this form, with fiber given by the $z$ direction.\footnote{There are more possible gluings for pseudo-Anosov flows, see \cite{BarbFen_pA_toroidal,BF_totally_per} for examples.}

\begin{figure} 
   \labellist 
    \small\hair 2pt
         \pinlabel $\alpha_1$ at 10 120
     \pinlabel $N_1$ at 110 145
   \pinlabel $\alpha_2$ at 175 155
     \pinlabel $\gamma$ at 115 65
 \pinlabel $N_2$ at 180 210
   \pinlabel {\em in} at 90 187
     \pinlabel {\em in} at 132 210
        \pinlabel {\em out} at 130 25
          \pinlabel {\em out} at 232 215 
 \pinlabel $\alpha_1$ at  370 90
 \pinlabel $\alpha_2$ at 460 110
  \pinlabel $\gamma$ at 405 80
 \pinlabel $L$ at 400 150
  \pinlabel $L'$ at 470 180 
    \endlabellist
\centerline{ \mbox{
\includegraphics[width=11cm]{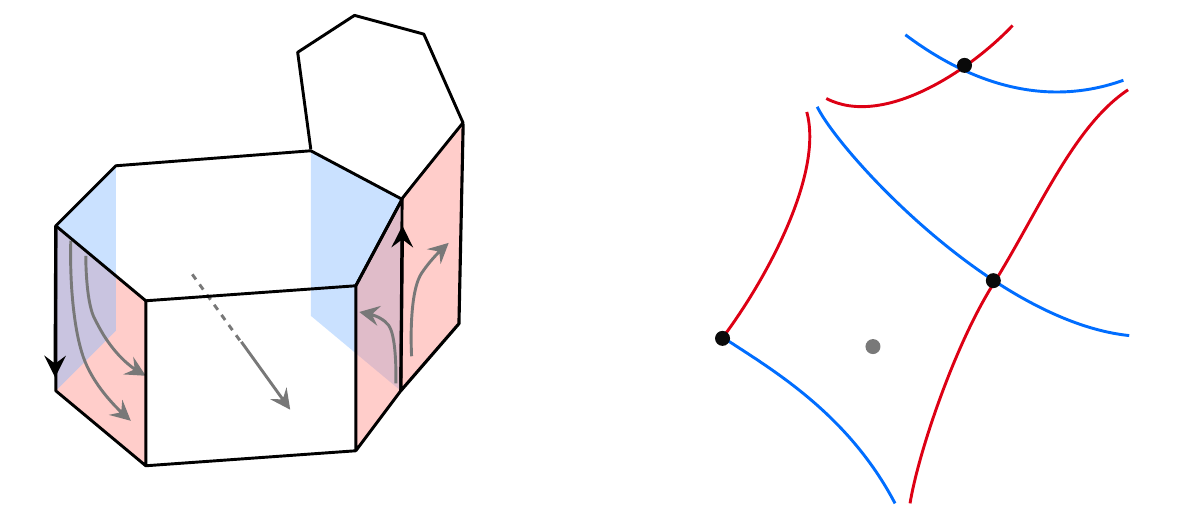}}}
\caption{{\em Left:} two model blocks glued along a half-face, the rear/right block is flipped vertically.  Orbits $\alpha_i$ are labeled according to the block $N_1$.  {\em Right:} lifting these blocks to $\wt{M}$ and projecting to $\orb_\phi$ gives two adjacent lozenges with corners fixed by the element representing the fiber.  In this image, the flow direction is out of the page, towards the reader.}
\label{fig:gluings}
\end{figure}
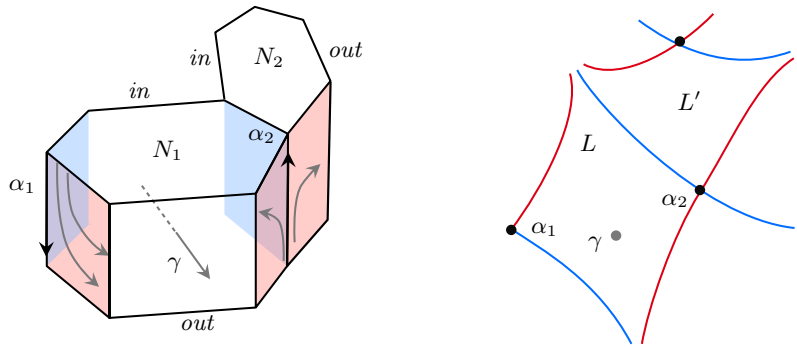 

\begin{construction}[Gluing model blocks around a nonsingular fiber]  \label{const:regular_glue}
Let $N_1, \ldots N_4$ be copies of $N$ with flows $\psi^+_1, \ldots \psi^+_4$ copies of $\psi^+$.   Let $(x_i, y_i, z_i)$ denote the coordinates on $N_i$. 

Glue the half-face $\{-\pi/2\} \times [-\pi/2, 0] \times S^1$ of $N_1$ to the half-face $\{\pi/2\} \times [-\pi/2, 0] \times S^1$ of $N_2$ by the map $(y_1, z_1) \mapsto (y_2, - z_2)$.   Glue the adjacent half-face $\{\pi/2\} \times [0, \pi/2, ] \times S^1$ of $N_2$ to the half-face $\{-\pi/2\} \times [0, \pi/2] \times S^1$ of $N_3$ by $(y_2, z_2) \mapsto (y_3, -z_3)$, the adjacent half-face $\{-\pi/2\}  \times [-\pi/2, 0] \times S^1$ of $N_3$ to $\{\pi/2\} \times [-\pi/2, 0] \times S^1$ in $N_4$, and the remaining half-faces of $N_4$ and $N_1$, each time by flipping the $z$-coordinate.  
The flows $\psi_1$ glue together, the result is a model neighborhood of a regular, periodic orbit corresponding to $\{-\pi/2\} \times \{0\} \times S^1$ in $N_1$.  The glued manifold has an obvious product structure, a trivial circle bundle over a surface.  
\end{construction} 

\begin{construction}[Gluing around an isolated singular Seifert fiber]  \label{const:cone_glue}
Let $N_1$, $N_2$ be copies of $N$ with flows $\psi_1^+$ as above, and glue $\{-\pi/2\}  \times [-\pi/2, 0] \times S^1$ of $N_1$ to the half-face $\{\pi/2\} \times [-\pi/2, 0] \times S^1$ of $N_2$ as above.   
Now glue the adjacent half-faces $\{-\pi/2\} \times [0, \pi/2] \times S^1$ of $N_1$ and $\{\pi/2\} \times [0, \pi/2] \times S^1$ of $N_2$ by the map
$(y_1, z_1) \mapsto (y_1, -z_1+1/2)$, where the $z$ coordinate is taken mod 1.  
Again, the flows glue together.  The $z$-coordinate direction gives a Seifert fibration on the resulting manifold with boundary, with isolated singular fiber given by the glued periodic orbit.  
\end{construction}

\begin{construction}[Gluing a reflector arc]  \label{const:arc_glue}
Let $N_1$, $N_2$ be copies of $N$ with flows $\psi_1^+$ and glue $\{-\pi/2\} \times [-\pi/2, 0] \times S^1$ of $N_1$ to the half-face $\{\pi/2\} \times [-\pi/2, 0] \times S^1$ of $N_2$ as above.    
Now, glue the half-face $\{-\pi/2\} \times [0, \pi/2] \times S^1$ of $N_1$ to itself by $(y_1, z_1) \mapsto (y_1, z_1 + 1/2)$, and similarly glue the half-face $\{\pi/2\} \times [0, \pi/2] \times S^1$ of $N_2$ to itself by  $(y_1, z_1) \mapsto (y_1, z_1+1/2)$.  This commutes with $X^+$ so gives a flow on the quotient manifold.  This manifold has a Seifert fiber structure with an arc of singular fibers.  
\end{construction} 

\begin{remark} \label{rem:plus_or_minus}
Note that all three gluing constructions can be performed with $X^-$ and $\psi^-$ used in place of  $X^+, \psi^+$, giving well-defined local model flows.  
\end{remark} 

\begin{definition}[Model flows and their flips]
We call a {\em model flow} any flow on a Seifert piece $P$ obtained by gluing blocks $N, \psi^+$ (or blocks $N, \psi^-$) as in 
Constructions \ref{const:regular_glue} - \ref{const:arc_glue}.    If $\phi^+$ is obtained by gluing copies of $N, \psi^+$, we denote by $\phi^-$ the flow obtained by the {\em same} gluings of $N$ but with $\psi^-$ instead of $\psi^+$ on each copy.  We call $\phi^-$ the {\em flip} of $\phi^+$, and refer to $\phi^\pm$ as a flipped pair.  
\end{definition}

\begin{proposition}[Periodic pieces are represented by models] \label{prop:nice_gluing} 
Let $P$ be a scalloped periodic Seifert piece of an Anosov flow $\phi$ in good position (equivalently, since all boundary surfaces are transverse in this case, represented by a small neighborhood of its spine). 
Then $(P, \phi)$ is orbit equivalent to a flow obtained by gluing copies of $(N, \psi^+)$ around periodic orbits as in Constructions \ref{const:regular_glue} -- \ref{const:arc_glue}.
\end{proposition} 

We call such a flow obtained by gluing the model blocks a {\em model flow} on $P$.  

\begin{proof} 
The proof uses the description of spines for scalloped periodic Seifert pieces.  
Let $P$ be a scalloped periodic Seifert piece of flow $\phi$.   We need to put good coordinates on $P$ to identify it with glued copies of $N$.  

Recall from the proof of Theorem \ref{thm:spines_exist} that $\wt{P}$ is homeomorphic to a small neighborhood of the lift $\wt{B}$ of a weakly embedded Birkhoff surface $B$, the lift $\wt{B}$ is homeormophic to $\cT' \times \R$, and $\wt{P}$ is homeomorphic to $\Sigma' \times \R$, where $\cT'$ is the (pruned) tree, and $\Sigma'$ is a surface giving $\cT'$ a fatgraph structure.  
The vertices of $\cT'$ correspond exactly to the lifts of the periodic orbits, and edges are the lifts of Birkhoff annuli.   
We may give $\Sigma'$ a $\pi_1(P)$-equivariant simplicial structure as follows: take one 2-cell for each edge of $\cT'$, bounded by $6$ one-cells, such that the edges adjacent to vertices of $\cT'$ project to local stable and unstable manifolds of the periodic orbits as in Figure \ref{fig:cells}.  
\begin{figure} 
   \labellist 
    \small\hair 2pt
     \pinlabel $\cT$ at 5 115
    \pinlabel $\Sigma'$ at 45 95 
    \endlabellist
\centerline{ \mbox{
\includegraphics[width=7cm]{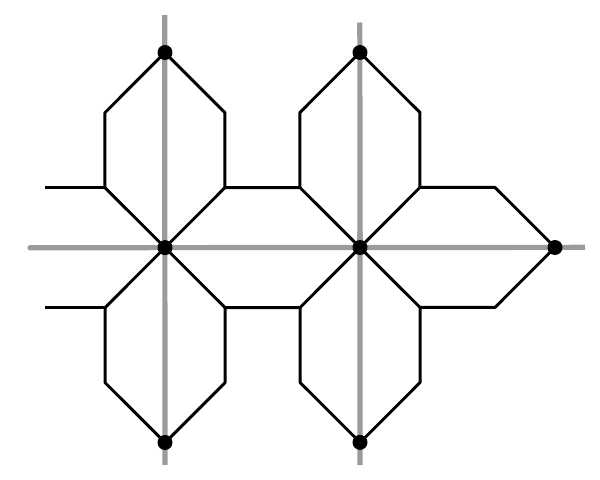}}}
\caption{Local structure of the 4-regular tree $\cT'$ (grey) and surface $\Sigma'$, indicating how to glue blocks.  The product of each 2-cell with $\R$ can be thought of as $\wt N$, where $\cT' \times \R$ is the lift of the Birkhoff annulus in the cell.}
\label{fig:cells}
\end{figure} 
This means that for each 2-cell $c$, the product $c \times \R$ with vector field generating the flow is homeomorphic to a copy of the universal cover $\wt{N}$ of the block $N$ with the lift of $X^+$.  The faces of $\wt{N}$ (including the half-faces bounded by periodic orbits) correspond to the boundary faces of $c \times \R$.     

We now describe how to put convenient global coordinates on $\wt{P}$ so that we can explicitly realize it as a union of copies of $\wt{N}$, with deck group acting so that the quotient is locally modeled on the gluings from constructions  \ref{const:regular_glue} --  \ref{const:arc_glue}. 
First, remove small disjoint fibered neighborhoods of each singular fiber,  and let $P_0$ denote the resulting space.  Then $P_0$ is a $S^1$ bundle over a surface with boundary.  Let $\sigma$ denote a section for this bundle structure, and think of $\sigma$ as embedded in $P$.

Lift $\sigma$ to a surface $\wt{\sigma}$ in the universal cover $\wt{P} \cong \Sigma' \times \bR$.   
Put coordinates on $\R$ so that $\wt{\sigma} \subset \Sigma' \times \{0\}$ and so that the action of an element $g$ of $\pi_1$ representing the regular fiber translates vertically by 1. 
 The fundamental group of $P$ is generated by $\pi_1(\sigma)$, the regular fiber, and for each removed neighborhood of an isolated singular fiber, appropriate generators of its fundamental group (in our case, we will soon see that the only singularities of the quotient orbifold are isolated cone points and reflector arcs, so these neighborhoods will all be solid tori or solid Klein bottles and have fundamental group $\bZ$).  See e.g. \cite{Scott} for background on generalized Seifert fiber structures, including the non-orientable setting.
 
 Adjusting our coordinates on $\R$ if needed, we may assume that all generators act by fiberwise isometries, i.e., maps of the form $(p, z) \mapsto (h_1(p), h_2(z))$ where $h_1$ is an automorphism of the pair $\cT', \Sigma'$ and $h_2$ an isometry of $\R$.  
The automorphisms of $\cT', \Sigma'$ must further send edges representing local stable (respectively, unstable) manifolds to local stable (resp. unstable) manifolds.  This means that any singular points of the orbifold given by the space of fibers in the Seifert fiber structure on $P$ are order 2 cone points or reflector arcs.   
Thus, the generators of neighborhoods of singular fibers 
  act by compositions of automorphisms of $\Sigma'$ and translations $z \mapsto z \pm 1/2$ in the vertical direction.

For each $2$-cell $c$ of $\Sigma'$, identify $c \times \bR$ with $\wt{N} = (I \times I) \times \bR$ respecting the product structures as follows.  
First, fix an identification of $c$ with $I \times I$ preserving orientation and sending $\cT' \cap c$ to $[-\pi/2,\pi/2] \times \{0\}$ and incoming/outgoing and stable/unstable manifold edges of $c$ to the corresponding edges of $I \times I \times \{0\}$ in $N$.  This defines a map in the $(x,y)$ coordinates on $I \times I$, extend this to be either the identity or $z \mapsto -z$ on the third coordinate, depending on the orientation of the vertical orbits in $c \times \bR \subset \wt P$ so that they match the image.   We may choose the $(x,y)$ coordinate map to be equivariant with respect to the action of $\pi_1(P)$ on $\Sigma'$ by first specifying it on a fundamental domain.  Using this identification, the lift of the vector field $X^+$ to $\wt{N}$ agree on the glued faces (as does $X^-$).  

Thus, we have a copy of $\Sigma' \times \bR \cong \wt{P}$ with globally defined vector field obtained by glued copies of $X^+$ (or by $X^-$) that descends to the quotient by the action of $\pi_1(P)$.  The induced flow on the quotient manifold, by construction, has the same spine as $\phi$ on $P$.  Again, by construction, the local model in the quotient about a vertex of $\cT$  is exactly one of those described in Constructions \ref{const:regular_glue} -- \ref{const:arc_glue}.  
\end{proof} 

\begin{corollary} 
Let $\phi$ be an Anosov flow and $\phi_P$ the restriction of $\phi$ to a scalloped periodic Seifert piece $P$, assumed in good position.  Then $\phi_P$ has a flip, i.e., a flow on $P$ which has the same spine but periodic orbits in opposite direction.
\end{corollary}

\begin{proof}
By Proposition \ref{prop:nice_gluing}, up to orbit equivalence $\phi_P$ is obtained by gluing copies of $N, \psi^+$ as in \ref{const:regular_glue} -- \ref{const:arc_glue}.   Let $\phi^-_P$ be the flow of the vector field obtained by performing the exact same gluings of copies of $N$, but replacing $X^+$ with $X^-$ in each copy.  This defines a flow on $P$ by Remark \ref{rem:plus_or_minus}.  
By construction $\phi^+_P$ and $\phi^-_P$ have the same spine, but the direction of each periodic orbit is reversed.  
\end{proof} 

\begin{rem}\label{rem_orbits_between_boundary}
Notice that the changed directions of periodic orbits of $\phi^+_P$ and $\phi^-_P$ are the only dynamical differences between the model and its flip in the following sense: there exists an orbit of $\phi^+_P$ between two boundaries of $P$ if and only if there also exists one for $\phi^-_P$. Moreover, these orbits are freely homotopic relative to boundary, i.e., this property holds even for lifts of boundary surfaces in the universal cover of $P$.
\end{rem}

While we have shown how to ``flip" a piece, we have {\em not} yet shown that a flip $\phi^-_P$ can be glued back into the manifold $M$ to produce an Anosov flow.   This is the goal of the next section.

%%%%%%%%%%%%%%%%%%%%%%%%%%%%%%%%%%%%%%%%%%%
\section{Building Periodic Seifert flips by gluing} \label{sec:build_flips}
In this section we first show that we can glue transverse boundary components of a model $(P, \phi^\pm_P)$ to each other tol obtain  {\em hyperbolic plugs} in the sense of \cite{BBY} (recalled below), provided that the parameter $\lambda$ is chosen sufficiently large.   
We then show how to use this structure to perform a Seifert flip of a flow on a general 3-manifold that admits a scalloped periodic Seifert piece, proving Theorem \ref{thm:existence}.  The uniqueness of flips (Theorem \ref{thm:flip_unique}) will be proved in Section \ref{sec:per_orbits}.

The reason for this two-step approach is because 
a Seifert piece $P$ of a manifold can have some of its boundary components glued to each other, and the others glued to other pieces.  Thus, there are two types of gluings that we need to do in order to insert model flows and flips into a manifold.  

%%%%
\subsection{Hyperbolic plugs and BBY-gluing}

Recall the definition of \emph{hyperbolic plugs}:
A hyperbolic plug is a $3$-manifold with boundary $V$ equipped with an Anosov flow $\phi$ that is transverse to the boundary and such that its maximal invariant set $\Lambda$ is non-empty and hyperbolic.
Notice that the stable manifold of $\Lambda$ induces a lamination $\cL_\phi^s$ on the incoming boundary $\partial_{in} V$ of $(V,\phi)$. Similarly, the unstable manifold of $\Lambda$ induces a lamination $\cL^u_\phi$ on the outgoing boundary $\partial_{out} V$ of $(V,\phi)$.

The model flows constructed in the previous section are already examples: 
\begin{lemma}
Any model flow $(P, \phi^+)$ or $(P, \phi^-)$ on a Seifert piece is a hyperbolic plug.  Moreover, if $\phi^\pm$ is a flipped pair, then $\cL_{\phi^+}^s = \cL_{\phi-}^s$ and the unstable boundary laminations agree also. 
\end{lemma}

\begin{proof} 
By construction the only orbits of $\phi^\pm$ that are completely contained in $P$ are the vertical orbits, i.e., those in the fiber (or $z$-coordinate) direction.   The union of periodic orbits is the maximal invariant sets for each flow, so $\phi^\pm$ are already hyperbolic on their maximal invariant sets.  The stable and unstable manifolds of the periodic orbits agree for $\phi^+$ and $\phi^-$, since, as noted in Observation \ref{obs:properties_model_flip}, this is true of the flows $\psi^\pm$ on $N$, and our gluing preserves stable and unstable.  Thus, the boundary laminations of $\phi^+$ and $\phi^-$ agree.  
\end{proof} 

We would like to also retain the property of being a plug after possibly gluing some boundary components of these pieces together.  The following lemma shows that this is indeed the case. 

\begin{lemma}\label{lem_partial_gluing_periodic_piece}
 Let $\phi_P$ be a model flow on a scalloped periodic Seifert piece $P$. 
 Let $T\in \partial_{out}P $ and $T'\in \partial_{in}P $ be a union of tori or Klein bottles in the boundary of $P$ on which the flow is, respectively, outgoing and incoming.  Let $h\colon T \to T'$ be a map which is linear on each component, in the sense that the derivative of $h$ is constant in the $(x,z)$-coordinates coming from each model block. 
Assume no closed leaf of $h(\cL^u|_T)$ is homotopic in $T'$ to a closed leaf of $\cL^s|_{T'}$. Then the flow induced by $\phi_P$ on $\bar{P} = P/(T\sim_h T')$ is a  hyperbolic plug.
\end{lemma}

This result is proven in \cite{BF_totally_per} when the manifold is orientable and for the case that {\em all} the incoming tori are glued to the outgoing tori, in which case one obtains a true Anosov flow instead of just a hyperbolic plug. The same strategy works in our slightly more general case, we give a sketch for convenience, using slightly different arguments.  

\begin{proof}
First we define a surface of section and we set some notation.  For each orbit representing a fiber in $P$ 
take a small closed disc centered on a point of the orbit, transverse to the flow, chosen such that these discs are pairwise disjoint.  Let $\Sigma$ denote the union of these transverse discs.  Thus, $\Sigma$ is a compact surface with the property that every orbit in the maximal invariant $\Lambda$ intersects the interior of $\Sigma$, and such that the return time is uniformly bounded --- i.e., it is a {\em local section} of $\phi_P$ in $P$. 
Let $\bar{P} = P/(T \sim_h T')$, let $\pi \colon P \to \bar{P}$ denote the quotient map, and let $\bar{\phi}$ denote the induced flow on $\bar{P}$.  
By construction,  $\Sigma \cup \pi(T)$ is a surface of section of $\bar{\phi}$ on $\bar{P}$. 

Let $f \colon \Sigma \cup \pi(T) \to \Sigma \cup \pi(T)$ be the first return map of $\bar{\phi}$. It is standard that showing uniform hyperbolicity of the map $f$ is enough to show hyperbolicity of the flow $\bar{\phi}$ on its maximal invariant.
The flow $\phi_P$ also gives partially defined ``return" or hitting maps  between subsurfaces, which we denote as follows.  
Let $f_{out, in} \colon T' \to T$ (defined only on the complement of the stable foliation of the vertical orbits) be the first return or hitting map of orbits of $\phi_P$ from $T'$ to $T$, and similarly let also $f_\Sigma\colon \Sigma \to \Sigma$, $f_{\Sigma, in} \colon T' \to \Sigma$ and $f_{out, \Sigma} \colon \Sigma \to T$ denote the hitting maps between these subsurfaces, restricted to the domains on which they are defined.  
Note that $f$ lifts to a map in $P$ that piecewise agrees with either $f_\Sigma$, $f_{\Sigma, in}$, $h \circ f_{out, \Sigma}$ or an appropriate restriction of $h \circ f_{out, in}$.
 
With this setup, the proof consists in showing that, for a large enough choice of parameter $\lambda$ in the definition of the model flow, the map $h \circ f_{out, in}$ is hyperbolic. Since the vertical (i.e., fiber direction) orbits are also hyperbolic, taking a small enough section $\Sigma$ will then give the hyperbolicity of $f$ and ends the proof. 
 
Hyperbolicity of $h \circ f_{out, in}$ is proved in \cite[Sec.~8 pp.~1944--45]{BF_totally_per} with an argument that applies directly here. In brief, by  using the $(x,z)$-coordinates given by the model flow from $N$ on each connected component of the domain of $f_{out,in}$ in $T'$, one shows that a good choice of unstable cone is one centered around $h_\ast(\partial/\partial z)$, and thin enough so that it does not contain $\partial/\partial z$. Similarly, a good stable cone is centered around $h ^{-1}_\ast(\partial/\partial z)$ and does not contain $\partial/\partial z$. It is then easy to see that for $\lambda \gg 1$ the expansion in the unstable cone under the map $h \circ f_{out, in}$ is as large as we want, and same for the stable cones under the inverse.
\end{proof}

%%%%%%%%%%%%%%%%%%%
\subsection{Existence of flips: Proof of Theorem \ref{thm:existence}} \label{sec:existence}
We return now to our main goal.  Namely, we suppose $M$ is a closed 3-manifold with Anosov flow $\phi$ and scalloped periodic Seifert piece $P$.   In the previous section, we constructed a pair of model flows $\phi^{\pm}_P$ on $P$ (with entering/exiting boundary) with the same spine as $\phi$ on $P$, and we now want to glue these together with the restriction of $\phi$ to $M \smallsetminus P$ to produce a pair of flows on $M$.  
We have also shown that $\phi^\pm_P$ are  \emph{hyperbolic plugs} in the sense of \cite{BBY}.  The only issue to gluing them using the techniques of \cite{BBY} is that the condition of having \emph{filling laminations}, which is necessary to use Theorem 1.5 of \cite{BBY} directly, is generally not satisfied here. This issue will be sidestepped by being more careful about the use of \cite{BBY}, an idea suggested to us by Fran\c{c}ois B\'{e}guin.  By doing this, we prove the following, which combined with Remark \ref{rem:non_transitive_BBY} below immediately gives Theorem \ref{thm:existence}

\begin{proposition} \label{prop_existence_of_flip}
Let $\phi$ be a transitive Anosov flow on $M$ and $P$ a scalloped periodic Seifert piece.  Let $\phi^{\pm}_P$ be associated model flows on $P$ with the same spine as $\phi|_P$, as constructed in Proposition \ref{prop:nice_gluing}.
Then the restriction of the flow $\phi|_{M\setminus P}$ on $M\setminus P$ and $\phi^\pm_P$ on $P$ may be glued along boundary components to produce flows on $M$ that are Anosov, have the same spines as $\phi$ on $P$ and such that their restrictions to $M\setminus P$ are isotopic to that of $\phi|_{M\setminus P}$.
\end{proposition}

\begin{rem}[Non-transitive setting] \label{rem:non_transitive_BBY}
 As will become apparent in the proof, the only reason we ask $\phi$ to be transitive in Proposition \ref{prop_existence_of_flip} is in order to be able to apply Theorem 1.5 of \cite{BBY}, where the assumptions are that a hyperbolic plug needs to be both filling and saddle, in order for the gluing to lead to an Anosov flow. However, the assumption of having a saddle is not actually needed for \cite[Theorem 1.5]{BBY} to hold,
and a version of this (in a much more general context) is proved in forthcoming work of Neige Paulet \cite{Paulet}. 
Consequently, Proposition \ref{prop_existence_of_flip} holds also for non-transitive Anosov flows.
\end{rem}

\begin{proof}[Proof of Proposition \ref{prop_existence_of_flip}]
Let $\phi$ and $P$ be as in the statement of the proposition.  Let $Q = M\smallsetminus P$. We emphasize that $Q$ is not necessarily a piece of the JSJ decomposition of $M$. 
If $Q = \emptyset$ then the proposition already holds by Lemma \ref{lem_partial_gluing_periodic_piece}. The argument is exactly the same for $\phi^+_P$ and $\phi^-_P$ so we work only with $\phi^+_P$ from here on.  
  
Since $P$ is a scalloped Seifert piece, the boundary surfaces of $P$ and $Q$ are all transverse to $\phi$.  As above, we denote by $\partial_{in}Q$, $\partial_{out}Q$ the boundaries of $Q$ where the flow respectively enters or exists $Q$. In the piece $P$, we further subdivide the boundary components into four classes: $\partial_{in}^Q P$, $\partial_{out}^Q P$ which are the boundaries where the flow $\phi$ goes from $Q$ to $P$ and from $P$ to $Q$ respectively, and $\partial_{in}^P P$, $\partial_{out}^P P$ which are the incoming and outgoing boundaries where the flow goes from $P$ to $P$.  It is possible that $\partial_{in}^P P$ and $\partial_{out}^P P$ are empty (in which case the work from Lemma \ref{lem_partial_gluing_periodic_piece} is not needed in the proof).  Note that no boundaries of $Q$ are glued together.

For each pair $S, S'$ with $S \in \partial_{out} Q \cup \partial_{out}^Q P \cup \partial_{out}^P P$ and $S' \in \partial_{in}Q \cup \partial_{in}^Q P \cup \partial_{in}^P P$, such that $S$ and $S'$ represent the same surface in $M$, we let $h_{S,S'} \colon S \to S'$ denote the gluing map, i.e. the map which projects to the identity between the images of $S$ and $S'$ in $M$.

%We also record the gluing maps between the $\partial_{in}$ and $\partial_{out}$ boundaries given by the embeddings in $M$. That is for any $S \in \partial_{out} Q$ (or $S \in \partial_{out} P$)  and $S' \in \partial_{in}^Q P$, such that $S$ and $S'$ represent the same surface in $M$, we have a map $h_{S,S'} \colon S \to S'$ that projects to the identity between the images of $S$ and $S'$ in $M$, and similarly for corresponding pairs in the other incoming and outgoing boundaries. (Note that no boundaries of $Q$ are glued together).

First we treat the gluings between boundaries of $P$.  If $S$ and $S'$ are two such surfaces, 
% with glued boundaries of $P$.  Suppose $S$ and $S'$ represent the same surface in $M$.  
the definition of scalloped says that $\pi_1(S) \simeq \pi_1(S')$ has fundamental group generated by periodic orbits.  These correspond to closed unstable (respectively stable) leaves in $S$ and $S'$ (respectively).  Thus, $h_{S, S'}$ does not send the free homotopy class of any closed stable leaf to a closed unstable.   By Lemma \ref{lem_partial_gluing_periodic_piece}, we can glue $S$ to $S'$ using a map 
%can glue the surfaces of $\partial_{out}^P P$ to those in $\partial_{in}^P P$ in pairs using a map 
that is linear with respect to the coordinates $(x,z)$ coming from the model flow on each block $N$ described in section \ref{sec_model_flow}, which is in the isotopy class of $h_{S,S'}$; and doing this for all surfaces $S, S'$ in $\partial_{out}^P P$  and $\partial_{in}^P P$ (respectively) we can obtain a hyperbolic plug $\bar{P}, \bar{\phi}$ from $\phi^+_P$. Notice that for this new plug, $\partial_{in}\bar P = \partial_{in}^Q P$ and $\partial_{out} \bar P = \partial_{out}^Q P$.

Then, we glue the outgoing surfaces in $\partial_{out} Q$ for the flow $\phi_Q$ to the corresponding ones in $\partial_{in}\bar P = \partial_{in}^Q P$ for the flow $\bar\phi$ using the original gluing maps. Since the transversality of the laminations is preserved and there are no new recurring orbits, Proposition 1.1 of \cite{BBY} implies that, after this gluing, we obtain a new hyperbolic plug $R, \phi_R$.  

Finally, we glue $R, \phi_R$ to itself by gluing $\partial_{out}R =\partial_{out}\bar P $ to $\partial_{in}R =\partial_{in}Q$ again using the original associated maps $h_{S,S'}$. Since $\phi$ was assumed to be transitive,
$\phi|_Q$ does not contain any attractor nor repeller (and nor does $\phi|_P$ or $\bar \phi$, by construction). By construction, $\phi_R$ is isotopically equivalent to $\phi_Q$ on $Q$ and to $\bar \psi$ on $\bar{P}$, thus $\phi_R$ has no attractor or repeller basic sets.
So all we have to do to satisfy the hypotheses of Theorem 1.5 of \cite{BBY} is to show that the unstable laminations of $\phi_R$ on $\partial_{out} R$ are filling.
To this end, let $A$ be a connected component of the complement of the unstable lamination in a surface of $\partial_{out} R = \partial_{out}\bar P$. Then the backward orbit of any point through $A$ exit in some component of $\partial_{in}R =\partial_{in}Q$, so in particular must have crossed a cutting surface, call it $Z$, between $Q$ and $\bar P$. Since all of these surfaces have been glued in such a way that they are scalloped, the image of $A$ on $Z$ under the backwards flow is a rectangle bounded by two segments of leaves of the stable lamination in $Z$ and two segments of leaves of the unstable lamination. In particular, $A$ must be a \emph{strip} (i.e., bounded by two unstable leaves that accumulate onto two compact leaves, see \cite[Definition 3.11]{BBY}).  Since all the complements of the unstable laminations in $\partial_{out} R$ are strips, the unstable lamination is by definition filling in $\partial_{out} R$, and similarly for the stable lamination in $\partial_{in} R$. 
 
Thus \cite[Theorem 1.5]{BBY} applies and we obtain the manifold $M$ with an Anosov flow $\phi^+_M$ which is isotopically equivalent to $\phi$ on $Q$ and to $\phi_P$ on $P$.    The same construction applies with $\phi^-$ used in place of $\phi^+$, giving a flow $\phi^-_M$ on $M$.  
\end{proof}

\begin{rem} \label{rem:not_unique} 
A technicality of the gluing theorem of \cite{BBY} is that given a filling saddle hyperbolic plug $(V,X)$ and an appropriate\footnote{Appropriate here means \emph{strongly transverse}, i.e., the stable lamination on $\partial_{in} V$ together with the image by $h$ of the unstable lamination on $\partial_{out} V$ must extend to a pair of transverse foliations, see \cite{BBY} for more details.} gluing map $h\colon \partial_{out} V \to \partial_{in} V$, it is first necessary to isotope the flow $X$ and isotope the map $h$ in order to ensure that the resulting glued flow is Anosov (the precise notion used is called a \emph{strong isotopy of the triple} $(V,X,h)$ in \cite{BBY}).

Left open in \cite{BBY} is whether two Anosov flows obtained from different, but strongly isotopic, plugs are necessarily isotopically equivalent (see \cite[Question 1.7]{BBY}). This was answered positively in \cite{BeguinYu} for orientable manifolds when the resulting Anosov flows are transitive.

In our context, the necessity of that isotopy in \cite{BBY} means that, a priori, 
\begin{itemize}
\item there could be lots of different flows obtained as a periodic Seifert flip of an Anosov flow $\phi$, 
\item if $\psi$ is obtained from $\phi$ by a periodic Seifert flip on $P$ via Proposition \ref{prop_existence_of_flip}, and $\phi'$ is obtained from $\psi$ via a periodic Seifert flip on $P$, then $\phi'$ and $\phi$ could a priori be non isotopically equivalent, and 
\item in the construction of Proposition \ref{prop_existence_of_flip}, if we glue the model flow of $\phi|_P$ (without flipping) back to $\phi|_Q$, we might obtain a flow $\phi'$ that is inequivalent to the original $\phi$.
\end{itemize}
One might be able to adapt the arguments of \cite{BeguinYu} to our context, and show that all these constructions yield isotopically equivalent flows. Instead, we will {\em prove} here that these flows are isotopically equivalent by using the main theorem of \cite{BFM} and the result of the next section.
\end{rem}
	
%%%%%%%%%%%%%%%%%%%%%%%%%%%%%%

\section{Free homotopy data of periodic Seifert flips} \label{sec:per_orbits} 

In this section we show that performing a periodic Seifert flip on a flow does not change the set of free homotopy classes represented by periodic orbits.  We use this to show that transitive flows have a unique (up to orbit equivalence) flip for each scalloped periodic Seifert piece, completing the proof of Theorem \ref{thm:flip_unique}.  We also prove Theorem \ref{thm_classification}, the classification theorem.   

In addition to the work in the previous sections, these proofs use as input the results of \cite{BFM}, and in particular the notion of the {\em sign data } of a tree of scalloped regions, applicable to any pair of transitive pseudo-Anosov flows with common sets of free homotopy classes of periodic orbits. We review these notions in Section \ref{sec:sign}.

\subsection{Flips preserve free homotopy data}\label{sec:uniqueness_flip}

\begin{proposition}\label{prop_same_spectra}
Let $\phi$ be an Anosov flow with a scalloped periodic piece $P$.  Let $\psi$ be a periodic Seifert flip of $\phi$ on $P$ in the sense of Definition \ref{def:flip}. 
 Then $\fix(\phi)=\fix(\psi)$.  
\end{proposition} 

To prove this, we begin with an elementary lemma about orbits that cross scalloped surfaces.  

\begin{lemma} \label{lem:order_crossing_surfaces}
Let $\phi$ be any Anosov flow on a manifold $M$, and let $\mathcal{Z} = \{Z_i\}$ denote the set of all lifts to $\wt{M}$ of all scalloped cutting surfaces for $\phi$ in $M$. For distinct $Z_i,Z_j$, write $Z_i <_\phi Z_j$ if there is an (oriented) orbit of $\wt \phi$ from $Z_i$ to $Z_j$.  
Then the relation $<_\phi$ is transitive
\end{lemma}

\begin{proof}
The proof proceeds by translating this picture to the orbit space of $\phi$.  Each $Z_i$ corresponds to a scalloped region $U_i$ in $\orb_\phi$.  An orbit passes through $Z_i$ and $Z_j$ if and only if $U_i \cap U_j \neq \emptyset$; the orbit is represented by a point $x$ in the intersection. 
The structure of scalloped regions, specifically, that they are product-foliated regions bounded by families of nonseparated leaves, constrains their possible overlaps.  Specifically, if for two scalloped regions we have $U_i \cap U_j \neq \emptyset$, then there is a lozenge $L$ in $U_i$ such that 
$L \supset U_i \cap U_j$ and $U_j$ intersects the two sides of $L$ on the boundary of $U_i$.   In other words, either 
\begin{enumerate} 
\item $U_j$ is contained in the stable saturation of a lozenge of $U_i$, whose unstable sides are in $\partial U_i$, or 
\item $U_j$ is contained in the unstable saturation of a lozenge of $U_i$, whose stable sides are in $\partial U_i$. 
\end{enumerate}

We claim that the orbit is oriented from $Z_i$ to $Z_j$ in the first case, and has the opposite orientation in the second. 
To see this, suppose $\gamma$ is an orbit oriented from $Z_i$ to $Z_j$.   Then, there is some small $\varepsilon$-neighborhood of $\gamma$ in $\cF^s(\gamma)$ so that orbits in this neighborhood travel alongside $\gamma$ from $Z_i$ to $Z_j$.  Any positive ray in $\cF^s(\gamma)$ must eventually intersect the $\varepsilon$-neighborhood of any positive ray of $\gamma$. Therefore, we deduce that any point in $\cF^s(\gamma)\cap Z_i$ will have a forward orbit in $\cF^s(\gamma)\cap Z_j$.  Translating this statement into the orbit space, this means that $\cF^s(\gamma) \cap U_i \subset \cF^s(\gamma) \cap U_j$. Thus, $U_j$ is contained in the stable saturation of a lozenge $L$ of $U_i$, where the unstable sides of $L$ are contained in $\partial U_i$ (and intersect $U_j$). 
 
Having established the claim, the desired conclusion follows immediately: if $Z_i <_\phi Z_j$ and $Z_j <_\phi Z_k$, this forces $U_k$ to be contained in the stable saturation of a lozenge from $U_i$ whose unstable boundary leaves are contained in $\partial U_i$, hence $Z_i < Z_k$.  
\end{proof}

\begin{proof}[Proof of Proposition \ref{prop_same_spectra}]
Suppose $\psi$ is obtained from $\phi$ by a periodic Seifert flip on $P$.  Since (after some isotopy equivalence) $\psi$ and $\phi$ agree on $M \setminus P$, there is a bijective correspondence between unoriented homotopy classes of their periodic orbits contained in $M \setminus P$.  
Further, the only periodic orbits of either flow that stay in $P$ are those freely homotopic to the Seifert fiber, so we also have a correspondence for these. Hence to prove that $\fix(\phi)=\fix(\psi)$, we only have to show a correspondence between the orbits crossing $P$. 

Let $\mathcal{Z}$ denote the set of all lifts to $\wt{M}$ of all boundary cutting surfaces of $P$.  By assumption, these are all scalloped for both $\phi$ and $\psi$.  
Suppose $g \in \pi_1(M)$ is represented by a periodic orbit $\alpha$ of $\psi$ that crosses $P$.  Then there is a lift $\wt{\alpha}$ to an orbit of $\wt{\psi}$ in $\wt{M}$ such that $g$ translates along $\wt{\alpha}$.  Let $Z_1, Z_2, \ldots Z_n \in \mathcal{Z}$ be lifts of scalloped surfaces crossed, in order, by the oriented orbit $\wt{\alpha}$ such that $Z_n = g(Z_1)$.  We choose these sequentially so that $\wt{\alpha}$ does not cross any element of $\mathcal{Z}$ between $Z_i$ and $Z_{i+1}$, so we have $Z_1 <_\psi Z_2 <_\psi \ldots <_\psi Z_n$. 

Now we consider the flow $\phi$, and the associated order $<_\phi$ on $\mathcal{Z}$ defined in Lemma \ref{lem:order_crossing_surfaces}.  Consider any consecutive pair $Z_i, Z_{i+1}$, $1 \leq i \leq n-1$.  The projection of the segment of $\wt{\alpha}$ between $Z_i$ and $Z_{i+1}$ either lies in $P$ or in $M \smallsetminus P$.  If it lies in $M\smallsetminus P$, then (since $\phi$ and $\psi$ are isotopic on $M\smallsetminus P$) it corresponds to a segment of an oriented orbit (with the same orientation) of $\phi$ from $Z_i$ to $Z_{i+1}$, thus we have $Z_i <_\phi Z_{i+1}$.  If it lies in $P$, then since $\phi$ and $\psi$ are flips of each other on $P$, there exists an oriented orbit of $\phi$ that goes from $Z_i$ to $Z_{i+1}$ (see Remark \ref{rem_orbits_between_boundary}), 
and so again $Z_i <_\phi Z_{i+1}$.  By Lemma \ref{lem:order_crossing_surfaces}, we have $Z_1 <_\phi  \ldots <_\phi Z_n = g(Z_1)$.  Thus, letting $U_1$ denote the associated scalloped region in $\orb_\phi$, we have $U_1 \cap g U_1 \neq \emptyset$.  
As we noted in the proof of Lemma \ref{lem:order_crossing_surfaces}, this means that $g U_1$ is contained in the (stable or unstable) saturation of a lozenge of $U_1$, hence $g$ contracts the interval of (unstable or stable) leaves that pass through $U_1$, giving a leaf fixed by $g$ in $U_1$.  Similarly, we have $g^{-1} U_1 \cap U_1 \neq \emptyset$ and hence a leaf of the other foliation is also fixed, thus a fixed point for $g$ in $\orb_\phi$, representing a periodic orbit in the free homotopy class of $g^{\pm 1}$. 
Thus, $\mathcal{P}(\psi) \subset \mathcal{P}(\phi)$ and by a symmetric argument we have equality.   
 \end{proof}

%--------------------------

\subsection{Sign data for trees of scalloped regions}  \label{sec:sign}

The definition of sign is somewhat difficult to parse within the context of \cite{BFM} since it is stated in the process of defining a map between orbit spaces of two flows.  But the essence of the idea is simple: when two transitive pseudo-Anosov flows have the same free homotopy data, there is a natural, well defined map between dense subsets of their respective orbit spaces that sends lozenges to lozenges and $g$-invariant trees of scalloped regions to $g$-invariant trees of scalloped regions, but on any particular tree of scalloped regions, it may preserve or reverse the dynamics of $g$.   This binary information (preserve versus reverse) is the notion of same or opposite sign.    

Our next goal is to prove that periodic Seifert flips change the sign on the tree of scalloped region associated to the Seifert piece. To do that, we need to explain the definition of sign a bit more formally.  

\subsection*{Signs -- set-up.} Suppose $\phi_1$, $\phi_2$ are transitive pseudo-Anosov flows on a manifold $M$ with $\cP(\phi_1) = \cP(\phi_2)$.\footnote{By \cite[Proposition 2.38]{BFM}, assuming only one is transitive, the fact that $\cP(\phi_1) = \cP(\phi_2)$ implies the other is as well.}   
We assume that at least one flow has a tree of scalloped regions in its orbit space.  In particular, the flow is not $\bR$-covered, which puts us in the main setting of \cite{BFM}.

Let $\orb_i$ denote the orbit space of $\phi_i$ and let $\orb^{NC}_i$ denote the set of points in $\orb_i$ which are fixed by some nontrivial element of $\pi_1(M)$ and {\em not} the corner of any lozenge.  By \cite[Lemma 2.30]{BFM}, $\orb^{NC}_i$ is dense in $\orb_i$. 
Let 
\[ \cP^{NC}(\phi_i): = \{ g \in \pi_1(M) \mid \exists x \in \orb^{NC}_i \text{ with } gx =x \}. \]
By \cite[Proposition 4.14]{BFM}, 
we have $\cP^{NC}(\phi_i) = \cP^{NC}(\phi_2)$.   By Proposition \ref{prop:fixed_chain}, each $x$ in $\orb^{NC}_i$ is the unique fixed point of some nontrivial $g \in \pi_1(M)$. Thus, there is a well-defined bijection $H\colon \orb^{NC}_1 \to \orb^{NC}_2$ sending the unique fixed point of $g$ on $\orb_1$ to its unique fixed point on $\orb_2$.    By construction, this map is $\pi_1(M)$-equivariant, satisfying $H(gx) = gH(x)$.  

Lemma 5.25 of \cite{BFM} show that $H$ sends points inside a $g$-invariant lozenge to points in some $g$-invariant lozenge in $\orb_2$ and preserves the property of being in a common $g$-invariant lozenge.  Thus, for a lozenge $L$, we may speak unambiguously of its image $H(L)$, meaning the lozenge in $\orb_2$ whose intersection with $\orb^{NC}_2$ is $H(L \cap \orb^{NC}_1)$.   $H$ also preserves the property of pairs of points being in adjacent lozenges sharing a side. It is further shown that
$H$ extends continuously to a well-defined map on all intersections of leaves of non-corner points, with the same equivariance, and lozenge-preserving properties listed above, as well as preserving the property of points being on a common leaf of a foliation.  Up to reversing the direction of one flow, this extended map (which we also denote by $H$) sends leaves of $\cF^u(\phi_1)$ to those of $\cF^u(\phi_2)$ and leaves of $\cF^s(\phi_1)$ to those of $\cF^s(\phi_2)$, wherever it is defined.  

The main result of \cite{BFM} shows that, in the absence of any trees of scalloped regions, this map $H$ in fact extends to a homeomorphism $\orb_1 \to \orb_2$ that conjugates the actions of $\pi_1(M)$ on the two orbit spaces.  Moreover, in the case where there do exist trees of scalloped regions, only one type of discontinuous behavior may occur, which we explain now.

\subsection*{Signs as markers of discontinuity} 
Suppose that $L$ and $L'$ are adjacent lozenges in $\orb_1$, with corners fixed by $g$.  For concreteness, suppose the shared side of $L$ and $L'$ is in $\cF^u$ and the action of $g$ on this leaf is expanding (the other cases are completely analogous).  
Then $H(L)$ and $H(L')$ are adjacent lozenges in $\orb_2$, with corners fixed by $g$.  If $H$ {\em did} extend to a $\pi_1$-equivariant homeomorphism, then the shared side of $H(L)$ and $H(L')$ would be a leaf on which $g$ was expanding.  This is in fact typically the case, the only possible exceptions being if $L$ and $L'$ lie in a tree of scalloped regions.  

If $T$ is a tree of scalloped regions in $\orb_1$ with corners fixed by $g$, then (since $H$ preserves lozenges and adjacency of lozenges), $H(T)$ is a tree of scalloped regions in $\orb_2$ with corners fixed by $g$ as well.   If, for one pair of adjacent lozenges $L$ and $L'$ in $T$ as above, their images $H(L)$ and $H(L')$ have a shared side on which $g$ has the same dynamics as shared side of $L$ and $L'$, then \cite[Lemma 5.37]{BFM} implies that this also holds for {\em every} pair of adjacent lozenges.  Thus, the map induced by $H$ on lozenges in $T$ either globally respects or reverses the dynamics of $g$ on shared sides.  
In the first case, $T$ and $H(T)$ are said to have the {\em same sign}, and otherwise that they have {\em opposite sign}.

With this background, we can prove the following.  
\begin{proposition}[Flips change sign] \label{prop:flips_change}
Suppose that $\phi_2$ is a transitive Anosov flow obtained from $\phi_1$ by a periodic Seifert flip on scalloped piece $P$, and let $T$ be the tree of scalloped regions associated to $P$ for $\phi_1$.  Then $H(T)$ is the tree of scalloped regions associated to $P$ for $\phi_2$, and $\phi_1$ and $\phi_2$ have opposite signs on $T$ and $H(T)$.  
\end{proposition} 

\begin{proof} 
For the proof, we may assume that $P$ is in good position for both flows, the restrictions of $\phi_1$ and $\phi_2$ to $M \setminus P$ agree up to an isotopy, and using Proposition \ref{prop:nice_gluing}, that $\phi_1$ is represented on $P$ by a model flow $\phi_P^+$ and $\phi_2$ by its flip $\phi_P^-$.  
By proposition \ref{prop_same_spectra}, we have $\cP(\phi_1) = \cP(\phi_2)$.   As above, we let $\orb_i$ denote the orbit space of $\phi_i$. 

Let $g \in \pi_1(M)$ represent the fiber direction of $P$, and $T_i$ in $\orb_i$ the associated tree of scalloped regions with all corners fixed by $g$.   Note that $T_2 = H(T_1)$, since $T_2$ is the maximal chain of $g$-invariant lozenges in $\orb_2$. 

By the density of $\orb^{NC}_i$ explained above, we may find a lift $\gamma$ of a periodic orbit of $\phi_1$, fixed by $h \in \pi_1(M)$, so that the unique fixed point of $h$ lies in a lozenge $L$ of $T_1$.    Consider a lozenge $L'$ adjacent to $L$ in $\orb_1$ and intersecting $\cF^u(\gamma)$.   Let $\wt{N'}$ be the block in $\wt{P}$ which projects to $L'$.  Since the action of $\pi_1(P)$ on $\wt{P}$ is cocompact, we may find $k \in \pi_1(P)$ such that $L, L'$ and $k(L)$ lie in a line of lozenges, in that order.  In other words, there are distinct lozenges $L = L_0, L'= L_1, \ldots L_m = k(L)$ with $L_i$ sharing a side with $L_{i+1}$ and intersecting $\cF^u(\gamma)$.  
By $\pi_1$-equivariance of the map $H$ and the fact that it preserves lozenges, stable and unstable leaves, and adjacency, we have that $H(L)$ is the lozenge in $\orb_2$ containing the unique fixed point of $h$, that $H(kL) = k H(L)$, and that $H(L')$ is adjacent to $H(L)$ and between $L$ and $kH(L)$ on a line of lozenges all intersecting $\cF^u(H(\gamma))$.  

We now translate this picture back to $\wt{P} \subset \wt{M}$.  The lozenges $L$ and $L'$ (or any adjacent pair along the line) are the projections to $\orb_1$ of lifts of adjacent Birkhoff annuli in $\wt{P}$ contained in adjacent model blocks for $\phi_1$ in $P$, as in Figure \ref{fig:gluings}, where $\gamma$ passes through the lift of a block associated to $L$.   Let $\wt{N}$ (the universal cover of the block labeled $N_1$ in the figure) denote the block which projects to $L$ in $\orb_1$, and $\wt{N'}$ the adjacent block which projects to $L'$.  Continuing along adjacent blocks glued along the stable manifolds of their periodic orbits, one eventually arrives at $k(\wt{N})$, which is separated from $\wt{N}$ by $\wt{N'}$.   

Suppose for concreteness that the direction of the element $g$ representing the fiber is the positive $z$ direction in the block associated to $L$, i.e., agrees with the orientation of the periodic orbit $\alpha_2$ in that block.  (This is the situation shown in Figure \ref{fig:gluings}, as in the figure, we use $\alpha_2$ to denote the (lifted) orbit where the blocks corresponding to $L$ and $L'$ are glued). 
Then, the action of $g$ is \emph{expansive} on $\cF^s(\alpha_2)$ in $\orb_1$, since its direction agrees with that of $\alpha_2$. 

Now, we consider $\phi_2$.  By construction, its restriction to $\wt{P}$ is constructed out of blocks glued in the same manner, and with the same action of $\pi_1(P)$, but with (the lift of) $\psi^+$ on each block replaced with $\psi^-$.  Thus, the directions of the periodic orbits in each block are reversed, but stable and unstable manifolds are preserved.  In particular, the action of $g$ is contracting on $\cF^s(\alpha_2)$.  
Projecting this picture back down to $\orb_2$ will now give us the desired conclusion: First observe that $\wt{N}$ projects to the unique lozenge $H(L)$  in $\orb_2$ containing $H(\gamma)$.  Since $k(\wt{N})$ is separated from $\wt{N}$ by $\wt{N'}$ and $H$ is $\pi_1(M)$-equivariant and preserves adjacency and lines of lozenges, it means that $k H(L) = H(kL)$ is separated from $H(L)$ by $H(L')$.  Thus, $H(L)$ and $H(L')$ share the side containing the projection of $\alpha_2$ to $\orb_2$\footnote{Establishing this fact is the reason we introduced the element $k$, using only the property that $H$ preserves adjacency, one cannot rule out he possibility that $H(L)$ shares the side of $L$ containing $\alpha_1$.}.
As we noted above, the action of $g$ is contracting on this stable side.  This, by definition, says the signs of $T_1$ and $T_2$ disagree. 
\end{proof}

%--------------------------------------------
\subsection{Proof of uniqueness of flips, non-isotopy equivalence, and classification of flows} 
Using the work of the previous subsection, we can now easily deduce that periodic Seifert flips are unique (up to isotopy equivalence), that they always change the isotopy-equivalence class of a flow, and prove our main classification result.  
 
\begin{proof}[Proof of Theorem \ref{thm:flip_unique} (Uniqueness of flips)]
Suppose $\phi$ is a transitive Anosov flow and $\psi$ and $\psi'$ are two flows, each separately obtained from $\phi$ by doing a periodic Seifert flip on a periodic Seifert piece $P$. 
By Proposition \ref{prop_same_spectra},  $\fix(\phi)=\fix(\psi) = \fix(\psi')$.  By  Proposition \ref{prop:flips_change} $\psi$ and $\psi'$ have opposite signs as $\phi$ on the tree of scalloped regions corresponding to $P$, so their signs agree.  
If $\phi$ has any other periodic Seifert pieces, $\psi$ and $\psi'$ also have the same signs (agreeing with that of $\phi$) on those.  Thus, $\psi$ and $\psi'$ have the same signs on trees of scalloped regions and the same set of unoriented free homotopy classes represented by periodic orbits, so by \cite[Theorem 1.3]{BFM} they are isotopically equivalent.  
\end{proof}

\begin{proposition} \label{prop:not_isotopy_eq}
Suppose that $\phi_1$ is a transitive Anosov flow on $M$, and $\phi_2$ is obtained from $\phi_1$ by a periodic Seifert flip.  
Then $\phi_1$ and $\phi_2$ are not orbit equivalent by any map isotopic to the identity.  
\end{proposition} 

\begin{proof} 
Since $\mathcal{P}(\phi_1) = \mathcal{P}(\phi_2)$, we have a $\pi_1(M)$-equivariant map $H$ from a dense subset of $\orb_1$ to a dense subset of $\orb_2$ as described above, where if $x$ is the unique fixed point of $g \in \pi_1(M)$ in $\orb_1$, then $H(x)$ is the unique fixed point of $g$ in $\orb_2$.  
If there is a homeomorphism $h$, isotopic to identity, that sends orbits of $\phi_1$ to that of $\phi_2$, then an appropriate lift $\wt h$ of $h$ will send each unique fixed point of an element $g$ of $\pi_1(M)$ in $\orb_1$ to its unique fixed point in $\orb_2$.  In other words, it will coincide with $H$ on a dense set.  Since $\wt h$ is continuous and $\pi_1$ equivariant, this implies that $H$ would have a continuous equivariant extension, contradicting the fact that flips have opposite signs (Proposition \ref{prop:flips_change}).  
\end{proof}

\begin{proof}[Proof of Theorem \ref{thm_classification} (Classification of flows)]
Let $\phi_1$ and $\phi_2$ be Anosov flows, at least one of which is transitive, and assume $\fix(\phi_1)=\fix(\phi_2)$. By \cite[Proposition 2.38]{BFM}, this implies they are both transitive.
By Theorem 1.3 of \cite{BFM}, either $\phi_1$ and $\phi_2$ are isotopy equivalent or there is at least one scalloped Seifert piece $P$ in $M$ on which $\phi_1$ and $\phi_2$ have a different sign (here we abuse terminology somewhat and say $P$ instead of ``the tree of scalloped regions corresponding to $P$").  Let $\psi_2$ be the flow obtained from $\phi_2$ by a periodic Seifert flip on $P$. 
By Proposition \ref{prop_same_spectra}, $\fix(\psi_2)=\fix(\phi_2)$, and by Proposition \ref{prop:flips_change}, the sign of $\psi_2$ now agrees with that of $\phi_1$ on $P$.  By construction the sign has not changed on any other scalloped periodic piece.  
Thus, applying a flip to each of the scalloped pieces on which signs differ, we obtain a flow $\psi$, such that $\fix(\psi)=\fix(\phi_1)$ and all signs agree, so by \cite[Theorem 1.3]{BFM}, $\psi$ and $\phi_1$ are isotopically equivalent.  
\end{proof}

%--------------------------------------------
\section{Orbit and isotopy equivalence classes of flips} \label{sec:equiv_flips}
In this section, we discuss how periodic Seifert flips change the orbit equivalence class of a flow.  We showed in Proposition \ref{prop:not_isotopy_eq} that applying a periodic Seifert flip always results in a flow that is not isotopy equivalent to the original flow.    But, it is possible that a flip does not change the orbit equivalence class (as we show in Example \ref{ex:seif_flip_flow}) due to some extra symmetries.   Nevertheless, by ruling out these symmetries one can construct large families of examples of non orbit-equivalent flows with the same free homotopy data.  We do this now, which will prove Corollary \ref{cor_sharpupperbound}.  
%%%%%%. 

\subsection{Building non-orbit equivalent examples}

Proposition \ref{prop:nice_gluing} says that periodic Seifert flows can be obtained by gluing model blocks.  For simplicity, we will only use the type of gluing in Construction \ref{const:regular_glue}, so only need to specify the combinatorial arrangement of the blocks to be glued together.  This is a construction already done in \cite[section 8]{BarbFen_pA_toroidal}, and we adopt their language and notation.  

\begin{definition}(see \cite{BarbFen_pA_toroidal})
An {\em admissible fatgraph} $X$  is a graph $X$ embedded in a surface $\Sigma$ that deformation retracts to $X$ that satisfies to the following properties:
\begin{enumerate} 
\item the valence of every vertex is even, and 
\item the set of boundary components of $\Sigma$ can be partitioned in two subsets (the outgoing and the incoming) so that for every edge $e$ of $X$, the two sides of $e$ lie in different subset of this partition.  
\item each loop in $X$ corresponding to a boundary component contains an even
number of edges.
\end{enumerate}
\end{definition} 
We will always take our examples to have all vertices of degree four.  

Given an admissible fatgraph $X$, one may associate a copy of a model flow $\psi^+$ on $N$  to each edge of the fat graph, and glue appropriate stable or unstable leaves of the vertical orbits according so that the outgoing (resp.~incoming) labels on the components of $\partial \Sigma$ correspond to the outgoing (resp.~incoming) annuli of $N$ as in Construction \ref{const:regular_glue}.  If all vertices have valence 4, the gluings around each vertex are exactly as in Construction \ref{const:regular_glue}; higher valence vertices give singular orbits of pseudo-Anosov flows.  

The result of the gluing is a compact Seifert manifold with boundary that is a circle bundle over $\Sigma$, with a flow that is incoming on some boundary tori and outgoing on others. Item (3) in the definition ensures that the boundary components are tori rather than Klein bottles, which makes it easier to work with.

We now build examples for the proof of Corollary \ref{cor_sharpupperbound}.  
We note that many examples of arbitrarily many non-orbit equivalent Anosov flows on the same manifolds have been built in the past (see \cite{BBY,BM,CP}), the point here is that we can build arbitrarily many that have the same free homotopy data.   Our procedure is quite general and can easily be adapted to give examples of pseudo-Anosov flows as well.  

We use as input families of admissible fatgraphs with the properties that they are pairwise non-homeomorphic surfaces, have at least two vertices, and each have one vertex that is preserved by all automorphisms.
A family of examples $X_n$ is described in Figure \ref{fig:special_graph} below, with the additional property that all vertices have valence 4, so that the associated flow has no singular orbits.  
\begin{figure} 
   \labellist 
    \small\hair 2pt
    \endlabellist
\centerline{ \mbox{
\includegraphics[width=7.5cm]{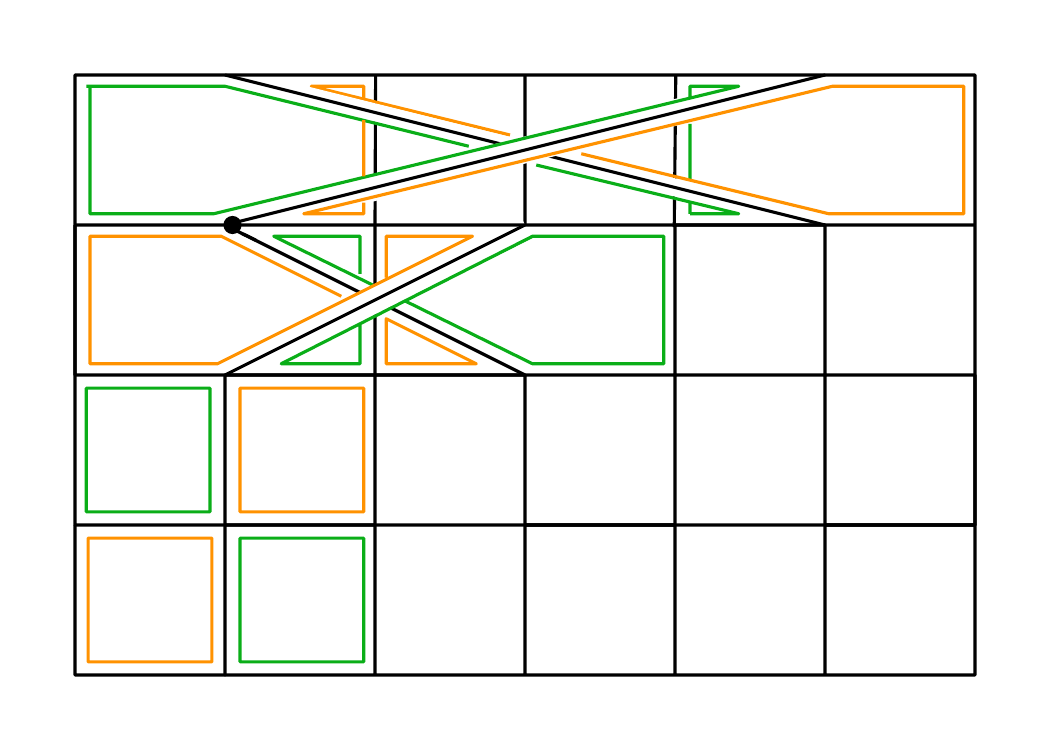}}}
\caption{An admissible fatgraph $X_1$ -- one identifies the top and bottom as well as the right and left sides by translation.  (For clarity, only some of the boundary components of the associated surface are shown).  The highlighted vertex is the unique vertex that is not adjacent to any ``quadrilateral'' boundary component (a boundary component following exactly 4 edges), thus invariant under all automorphisms.  
To produce a family of examples, take $X_n$ to be the result of replacing the 2 bottom rows of squares with 2n rows.}
\label{fig:special_graph}
\end{figure}

\begin{construction}  \label{const:gluing_for_flip}
Let $X_1, X_2, \ldots X_n$ be admissible fatgraphs constructed as in figure \ref{fig:special_graph}, with pairwise non-homeomorphic surfaces $\Sigma_1, \ldots, \Sigma_n$.  Thus each $X_i$ has a vertex $v_i$ that is preserved by all fatgraph automorphisms.  Let $P_i$ be the Seifert pieces with model flows $\lambda_i$ associated to $X_i$ as in \cite{BarbFen_pA_toroidal} and let $o_i$ denote the orbit associated to $v_i$. 

For each $i$, choose a boundary torus $T_i$ of $P_i$ that intersects the stable annulus of $o_i$.  The flows are all incoming on these tori.  Also, choose a boundary torus $T'_i$ of $P_i$ that intersect the unstable annulus of $o_i$, these are outgoing for the flows.   Glue $T_i$ to $T'_{i+1}$ (with indices taken modulo $n$). 
Glue the remaining boundary tori of the $P_i$ together in incoming/outgoing pairs, chosen so that, aside from the already specified $T_i$ and $T'_j$, no torus intersecting an unstable or stable annulus of any $o_i$ is glued to any other such torus. One has a great deal of flexibility in the gluing, the main property we need of the resulting manifold $M$ is that the $P_i$ are indeed Seifert pieces  of the JSJ decomposition (so one must not glue the fiber direction in one piece to that in another), and that the manifold is orientable, which will be important in the proof.  

Provided the constant $\lambda$ in the definition of the model flows has been chosen large enough (as in Lemma \ref{lem_partial_gluing_periodic_piece} above), 
with such a gluing, one easily verifies that the transitivity condition of \cite[Proposition 1.6]{BBY} is satisfied, so we obtain a totally periodic, transitive, Anosov flow $\phi$ on the resulting graph manifold $M$.   
\end{construction}

Now, to prove Corollary \ref{cor_sharpupperbound}, it suffices to prove the following.  
\begin{proposition} 
Let $M, \phi$ be obtained as in Construction \ref{const:gluing_for_flip}.  Suppose that $\psi_i$ is obtained from $\phi$ by a Seifert flip along $P_i$.   Then $\phi$ and $\psi_i$ are not orbit equivalent, and for any $i \neq j$, $\psi_i$ and $\psi_j$ are not orbit equivalent.  
\end{proposition} 

\begin{proof} 
Suppose first for contradiction that $f\colon M \to M$ was an orbit equivalence between $\psi_i$ and $\phi$.  (The strategy to show $\psi_i$ and $\psi_j$ are inequivalent is essentially the same, and we treat this at the end.) 
Since the surfaces $\Sigma_i$ are pairwise non-homeomorphic, the Seifert pieces $P_i$ are non-homeomorphic, so $f$ preserves each piece.  
Theorem D of \cite{BF_totally_per} (or, more simply, the construction and properites of spines) implies that $f$ induces an automorphism of each of the fat graphs $X_i$, hence preserves the special vertices $v_i$ and their associated orbits $o_i$.  
Thus $f$ preserves the isotopy classes of the tori $T_i$ in $M$, since these are characterized by their intersections with the stable and unstable annuli of the orbits $o_i$.  

By definition of Seifert flip, the direction of $\psi_i$ and $\phi$ differ on $o_i$ but agree on $o_j$ for $j \neq i$. 
Assume as a first case that $f$ is an oriented orbit equivalence, preserving the direction of the flow.  Then $f$ reverses the fiber direction of $P_i$, but preserves that of $P_j$ when $j \neq i$. 
  
Consider any $j \neq i$.  
Since $f$ preserves $T_j$, it preserves the orientation of the stable annulus of $o_j$, and similarly, it preserves $T_j'$ so preserves the orientation of the unstable annulus of $o_j$, and hence is orientation preserving on $P_j$.  
Similarly, since $f$ reverses the direction of the fiber on $P_i$, but preserves $T_i$ and $T_i'$, then it must be orientation reversing on $P_i$.  This is a contradiction since $M$ was assumed orientable.  

The case where $f$ reverses the direction of the flow is handled similarly; here $f$ must reverse the orientation of $P_j$ for $j \neq i$ but preserve that of $P_i$.

Finally, in order to show $\psi_i$ and $\psi_j$ are not orbit equivalent, one runs exactly the same proof, and shows that any candidate $f$ for an orbit equivalence between $\psi_i$ and $\psi_j$ would have to preserve orientation on $P_k$ for $k \neq i, j$ and reverse that on $P_i$ and $P_j$, or vice versa, depending on whether $f$ preserved or changed the direction of the flow. 
\end{proof}

\subsection{Orbit equivalent flips}
By contrast with the previous section, our next example shows that some nontrivial Seifert flips {\em do} produce orbit equivalent flows.

\begin{example}  \label{ex:seif_flip_flow}
Consider a 2-holed torus fatgraph as in Figure \ref{fig:flip_symmetry}.  This has a nontrivial symmetry that preserves orientation, preserves the two boundary components and exchanges the two vertices.   

\begin{figure}[h]
   \labellist 
    \small\hair 2pt
 \pinlabel A at -5 75
  \pinlabel A at 182 75
   \pinlabel B at 48 30
  \pinlabel C at 48 120
   \pinlabel C at 135 30
  \pinlabel B at  135 120
 \endlabellist
\centerline{ \mbox{
\includegraphics[width=11cm]{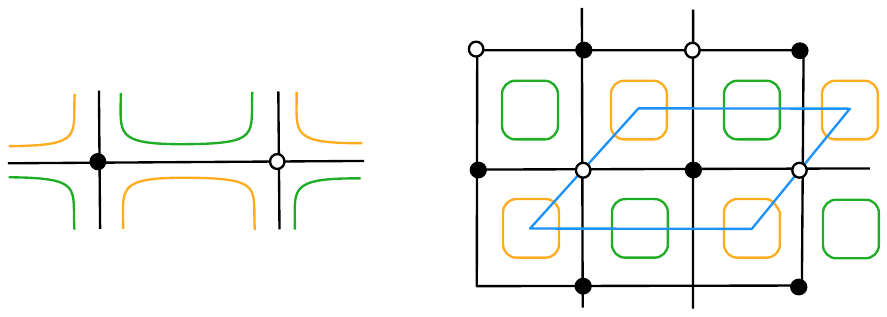}}}
\caption{A fatgraph with two vertices and four edges (with indicated gluings), seen also as a quotient of a fatgraph in $\bR^2$}
\label{fig:flip_symmetry}
\end{figure}
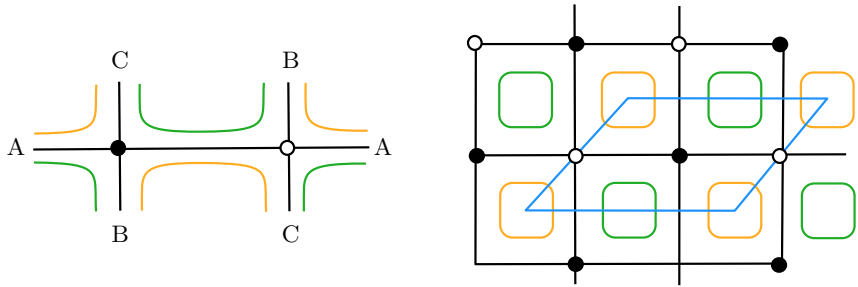 

One way to describe this concretely is as follows:  Let $\hat{\Sigma}$ denote $\bR^2$ with a small $\epsilon$-ball or square removed about every point of the lattice $\bZ^2$, and let $\hat{X} \subset \hat{\Sigma}$ be the fatgraph with vertices at the half-integer points $\bZ^2 + (1/2, 1/2)$ in $\hat{X}$, and horizontal and vertical edges between adjacent points.  Then the graph $X$ of Figure \ref{fig:flip_symmetry} is the quotient of $\hat{X}$ by the lattice $\Gamma$ of translations generated by $(2,0)$ and $(1,1)$, and the ambient surface $\Sigma$ is the quotient of $\hat{\Sigma}$.  We note in particular that the action of $\Gamma$ on $\hat{X}$ has two distinct orbits of vertices, hence the two vertices in the quotient.  

The linear map $(x, y) \mapsto (-y,x)$ on $\R^2$ has order 4, preserves the lattice $\Gamma$ and graph $\hat{X}$, exchanges the two $\Gamma$-orbits of vertices of $\hat{X}$, and preserves the $\Gamma$-orbits of boundary components of $\hat{X}$.  Thus, it descends to an order 4 automorphism of $X \subset \Sigma$ preserving boundary and exchanging the two vertices.

 Let $\sigma$ denote this symmetry of the associated Seifert fibered space $\Sigma \times S^1$, trivial on the $S^1$ fibers.  By construction, $\sigma$ is an orbit equivalence between the model flow $\phi^+$ on $\Sigma \times S^1$ and its flip $\phi^-$.

Now we consider $M$ the manifold obtained by gluing two copies of $\Sigma \times S^1$ along their boundary tori. 
We can obtain two flows on $M$, called $\phi_1$ and $\phi_2$ by gluing either two copies of $\phi^+$ or one copy of $\phi^+$ and one of $\phi^-$. By definition, $\phi_1$ is a periodic Seifert flip of $\phi_2$. 
Since $\sigma$ is isotopic to identity on the boundaries, one can construct a homeomorphism $\bar\sigma$ that realizes $\sigma$ on one copy of $\Sigma \times S^1$ and the identity on the other. Then the conjugate of $\phi_1$ by $\bar\sigma$ is also a flip of $\phi_1$. So by uniqueness of flips (Theorem \ref{thm:flip_unique}) $\phi_2$ is isotopically equivalent to the conjugate of $\phi_1$ by $\bar\sigma$, i.e., $\phi_1$ and $\phi_2$ are orbit equivalent.
\end{example}

\bibliographystyle{amsalpha}
\bibliography{scalloped_trees}

\end{document}